\documentclass[10pt]{article}
\usepackage{amsmath,tabu}
\usepackage{amssymb}
\usepackage{amsfonts}
\usepackage{amsthm}
\usepackage{slashed}
\usepackage{cancel}
\usepackage{mathtools}
\usepackage{paralist}
\usepackage{graphicx}
\usepackage{float}
\usepackage{tikz}
\usepackage{tikz-cd}
\usepackage{pgfplots}
\usepackage[hmargin = 3.7 cm,vmargin = 3.7 cm]{geometry}
\usepackage{hyperref}
\numberwithin{equation}{section}
\usepackage{cleveref}
\usepackage{cite}
\usepackage{array}
\usepackage{setspace}
\usepackage{bbm}
\usepackage{dsfont}

\usepackage{blkarray}

\usepgfplotslibrary{fillbetween}
\usetikzlibrary{calc}

\newtheorem{lemma}[equation]{Lemma}
\newtheorem{proposition}[equation]{Proposition}

\newtheorem{corollary}[equation]{Corollary}
\newtheorem{thm}{Theorem}

\theoremstyle{definition}
\newtheorem{definition}[equation]{Definition}

\theoremstyle{remark}

\newtheorem*{remark*}{Remark}

\theoremstyle{remark}

\theoremstyle{definition}
\newtheorem{problem}{Problem}

\newcommand{\Vol}{\mathrm{Vol}}
\newcommand{\im}{\mathrm{im}}

\newcommand{\SU}{\mathrm{SU}}

\newcommand{\SPAN}{\mathrm{span}}

\renewcommand{\Re}{\operatorname{Re}}
\renewcommand{\Im}{\operatorname{Im}}

\newcommand{\SL}{\mathrm{SL}}
\newcommand{\dlangle}{\langle\!\langle}
\newcommand{\drangle}{\rangle\!\rangle}
\newcommand{\dlpara}{(\!(}
\newcommand{\drpara}{)\!)}
\newcommand{\Alt}{\mathrm{Alt}}

\newcommand {\cL}{{\cal L}}
\newcommand{\bC}{{\mathbb{C}}}
\newcommand{\bP}{{\bf P}}
\newcommand{\db}{\overline{\partial}}
\newcommand{\bR}{{\mathbb{R}}}
\newcommand{\bZ}{{\mathbb{Z}}}
\newcommand{\hPsi}{\hat{\Psi}}

\newcommand{\Lv}{\mathcal{L}_v}

\newcommand{\KN}{\mathbin{\bigcirc\mspace{-15mu}\wedge\mspace{3mu}}}

\title{Closed 3-forms in five dimensions and embedding problems}

\usepackage[]{authblk}
\author[1,2]{Simon Donaldson}
\affil[1]{\footnotesize Imperial College London}
\author[2]{Fabian Lehmann}
\affil[2]{\footnotesize Simons Center for Geometry and Physics, Stony Brook University}

\begin{document}

\maketitle

In this paper we begin the study of an embedding question for $3$-forms on $5$-dimensional manifolds. Let $Z$ be a complex Calabi--Yau threefold, that is, a $3$-dimensional complex manifold with a nowhere-vanishing holomorphic $3$-form, which we write as $\Psi+ i\hPsi$ where $\Psi,\hPsi$ are real $3$-forms. From another point of view, $Z$ is a manifold with a torsion-free $\SL(3,\mathbb{C})$-structure.  Let $M$ be a $5$-manifold and $\psi$ a given  closed $3$-form on $M$. The question we consider is the existence of an embedding $F:M\rightarrow Z$ such that $F^{*}\Psi=\psi$. The main case we have in mind is when $Z$ is $\bC^{3}$ with its standard holomorphic $3$-form. 

Our original motivation for considering this question comes from Hitchin's approach to $\SL(3,\bC)$-structures \cite{Hitchin}. A special feature of the algebra of $3$-forms in dimension $6$ is that there is an open set of forms  $\Psi$  which  determine an almost-complex structure and the imaginary part $\hPsi$ algebraically. The structural equations are equivalent to the conditions that $\Psi$ and $\hPsi$ are closed $3$-forms: they can be viewed as a system of  partial differential equations for $\Psi$. On a closed $6$-manifold,  Hitchin gave a variational formulation in terms of the volume functional on forms $\Psi$ in  a given de Rham cohomology class. In this setting there is a natural boundary value problem on a $6$-manifold $Z_{0}$  with boundary $\partial Z_{0}=M$ and with a given $3$-form $\psi$ on $M$.  One seeks a solution $\Psi$ to the partial differential equations which restrict to $\psi$ on the boundary. This is the analogue, in dimension $6$, of the $7$-dimensional theory for $G_{2}$-structures studied in \cite{Donaldson-Bdry-G2}. We plan to develop this boundary value theory further in another article but, in the perturbative theory which we focus on in the body of the current paper, if $Z_{0}$ is, for example,  a pseudoconvex domain with smooth boundary in $\bC^{3}$,  then by results of Hamilton \cite{Hamilton1977Deformation2} any deformation of $Z_{0}$ as an abstract complex manifold with boundary is realised as a deformation of the domain within  $\bC^{3}$. Then the  boundary value problem in this setting is essentially equivalent to the embedding problem for $M$ in $Z$.

An informal count tells us that a closed $3$-form in five dimensions depends locally on $6$ unconstrained functions. (That is, we can write such a form as $d\tau$ for a $2$-form $\tau$, which gives $10$ functions: we can change $\tau$ to $\tau+d\eta$, so we subtract $5$, but if $\eta=df$ the change is ineffective so we add $1$,  and our count is $10-5+1=6$.)
Since a map from $M$ to a $6$-manifold also depends on $6$ functions the  count suggests that the embedding question is a reasonable one. This  is special to the dimension $5$: for $m>3$, on a manifold of dimension $2m-1$ the closed $m$-forms form a much \lq\lq larger''  space than the maps to a $2m$-manifold.
The situation is somewhat like that in Riemannian geometry where, at the level of function counting, in dimension $n=2$ it is reasonable to  seek an isometric embedding of an abstract Riemannian $n$-manifold as a hypersurface in $\bR^{n+1}$ , but not for $n>2$.   In that setting, the  solution (by Nirenberg and Pogorelov) of the famous Weyl problem gives existence and uniqueness for the case when the metric on the surface has positive Gauss curvature: the image is then the boundary of a convex domain in $\bR^{3}$. In a similar
 vein,  in this article we focus on data satisfying a  pseudoconvexity condition. We will also see (in subsection \ref{section-Minkowski}) that the Minkowski problem---another famous classical embedding problem---can be obtained as a dimensional-reduction of our theory.

We begin Section \ref{section1} with  an elementary  study of the structure of closed $3$-forms on $5$-manifolds $M$ and define an open set of {\it strongly pseudoconvex} forms. If a solution to the embedding problem exists these correspond to strongly pseudoconvex hypersurfaces in the ordinary sense of several complex variable theory. Such a  $3$-form defines a contact structure $H\subset TM$ and an orthonormal pair of $2$-forms $\omega,\alpha$ on $H$. A  solution of the embedding problem gives a third $2$-form $\beta$ on $H$ such that $\omega,\alpha,\beta$ make up an orthonormal triple, satisfying certain equations involving the exterior derivative $d$ and its restriction $d_{H}$ to $H$. This orthonormal triple defines an $\SU(2)$-structure, of a kind which we call \textit{contact hyperk\"ahler}. A special class of such structures is formed by the {\it Sasaki-Einstein} structures, which we define in \ref{section-Sasaki-Einstein}. Contact hyperk\"ahler $\SU(2)$-structures are a special type of so called \textit{nearly hypo} $\SU(2)$-structures which are induced on real hypersurfaces in $6$-dimensional manifolds with a torsion-free $\SU(3)$-structure \cite{Conti2005GeneralizedKS}. These in general come without a contact structure.

As Robert Bryant pointed out to us, the core of the embedding question can be formulated as a problem on the $5$-manifold $M$. Starting with a strongly pseudoconvex $3$-form $\psi$, and hence a pair $(\omega,\alpha)$, the problem is to extend this to a contact hyperk\"ahler structure $(\omega,\alpha,\beta)$. This is a nonlinear PDE for $\beta$ on the $5$-manifold which can be viewed as a \lq\lq contact version'' of the Calabi--Yau problem in four real dimensions ({\it i.e.} the existence of a hyperk\"ahler structure). If we have a solution $\beta$ the pair $(\alpha,\beta)$ defines a CR structure. From then on our embedding question becomes essentially the much-studied CR-embedding problem.

In the first part of Section \ref{section-Nash-Moser} we analyse the perturbative version of our problem, for deformations around a given solution,   and the associated linearised question. 
 The $\SU(2)$-structure
defines a Euclidean metric on $H$ and a decomposition of the $2$-forms on $H$ into self-dual and anti-self-dual parts: $\Omega^{2}_{H}=\Omega^{+}_{H}\oplus \Omega^{-}_{H}$, where the self-dual subspace is spanned by $\omega,\alpha,\beta$. We find that the key operator in the linearised theory is
$$   d^{-}_{H}:\Omega^{1}_{H}\rightarrow \Omega^{-}_{H}. $$
The relevant foundations from linear analysis are developed before, in Section \ref{section_linear_analysis}. In particular we show that the vector space
$$ {\cal H}= \{ \sigma\in \Omega^{-}_{H}: d_{H}\sigma =0\}, $$
can be identified with the cokernel of $d^{-}_{H}$. This space ${\cal H}$ thus appears as the obstruction to solving the deformation problem. We also show that the vanishing of ${\cal H}$ is an open condition.  The main feature of the linear analysis is that, as in the CR theory, the relevant operators are subelliptic, not elliptic, and the inverses suffer a \lq\lq loss of derivatives'' in Sobolev spaces. In the second part of Section \ref{section-Nash-Moser} we apply the Nash--Moser inverse function theorem to obtain our deformation result, in the case when ${\cal H}=0$. This requires a careful study of the dependence of the estimates for inverse operators on parameters.

Given $\psi$, an embedding $F:M\hookrightarrow Z$ with $F^*\Psi=\psi$ is in general not unique. If $F(M)$ is the boundary of a domain $U\subset Z$, and if $\Phi: \bar{U}\rightarrow Z$ is a diffeomorphism to its image which is holomorphic and satisfies $\Phi^*\Psi=\Psi$, then $\Phi\circ F$ is another embedding which realises $\psi$. If for example the ambient space is $\mathbb{C}^3$, then the restriction of every element in $\SL(3,\mathbb{C})$ is such a diffeomorphism.

We collect our results in the following Theorem. 

\begin{thm}
\label{mainTHM}
Let $(\theta,\omega,\alpha,\beta)$ be a contact hyperk\"ahler $\SU(2)$-structure on $M$. 
\begin{itemize}
\item
The space $\mathcal{H}$ is finite dimensional.
\item
Suppose the $\SU(2)$-structure is induced by an embedding $F:M\hookrightarrow Z$ and that $\mathcal{H}=0$. Then for every closed $3$-form $\tilde{\psi}$
in the de Rham cohomology class of
$\psi=F^*\Psi$ which is sufficiently close to $\psi$
there is an embedding $\widetilde{F}$ close to $F$ such that $\widetilde{F}^*\Psi=\tilde{\psi}$.
If the ambient space is $Z=\mathbb{C}^3$, then $\widetilde{F}$ in a neighbourhood of $F$ is unique up to holomorphic diffeomorphisms as above. 
\item
If the $\SU(2)$-structure is Sasaki--Einstein and $Z$ is Stein, then $\mathcal{H}=0$. In particular this is true for the standard embedding $S^5\hookrightarrow \mathbb{C}^3$.
\end{itemize}
\end{thm}

We prove the first item in the more general context of contact-metric $5$-manifolds.
For the the proof of the last item in Theorem \ref{mainTHM} we need some of the theory of the $\db_{b}$ complex,  which is reviewed in Appendix \ref{appendix-vanishing}.

We do not know any examples of contact hyperk\"ahler manifolds which bound a strongly pseudoconvex region in a Stein manifold for which ${\cal H}$ is non-zero. It is possible that it is always zero, which would greatly extend the scope of our result. In Appendix \ref{appendix-B} we give an explicit neighbourhood of the standard structure on $S^{5}$ where ${\cal H}$ vanishes and in Appendix \ref{appendix-C} we obtain a curvature criterion for vanishing, via a Weitzenbock formula. This involves  differential-geometric constructions which  have independent interest. 

The uniqueness statement in Theorem \ref{mainTHM} only covers the case when $Z=\bC^{3}$. We are confident that there is a similar statement in general but that seems to be more easily treated in the framework of the boundary value theory alluded to at the beginning of this introduction,  so we do not go into it here.

The authors are very grateful to Robert Bryant for many helpful discussions and suggestions related to this work.

This research was supported by the Simons Foundation through the Simons Collaboration on {\it Special Holonomy in Geometry, Analysis and Physics}.

\section{Closed $3$-forms in dimension five}
\label{section1}

\subsection{The structure of closed 3-forms on a five dimensional manifold}

\label{section-structure-in-dim5}

Let $M$ be an oriented $5$-manifold and $\psi$ a closed $3$-form on $M$.
In this section we describe the structure induced on $M$ by $\psi$ under a further convexity condition. The chosen terminology will become clear in the next section.

\begin{definition}
\label{Def-pseudoconvex}
$\psi$ is called \textit{strongly pseudoconvex} if
it satisfies the following three properties:
\begin{compactenum}[(1)]
\item
The skew-symmetric bilinear form on cotangent vectors with values in the real line of volume forms given by
\begin{align}
(\lambda,\eta) \mapsto \lambda\wedge\eta\wedge\psi
\label{bilinear}
\end{align}
has maximal rank, i.e. $4$, at each point.
This defines a rank 4 subbundle $H\subset TM$.

To describe the next two conditions, 
let $\theta$ be a $1$-form which at each point spans the $1$-dimensional space of cotangent vectors for which \eqref{bilinear} is degenerate. Then $\psi$ can be written as $\psi=\theta\wedge\alpha$ for some $2$-form $\alpha$ and $H$ is the kernel of $\theta$. 
\item
$\theta\wedge\alpha^2$ does not vanish anywhere on $M$.
\item
$\theta\wedge(d\theta)^2$ is a positive multiple of $\theta\wedge\alpha^2$.
In particular, $d\theta$ is non-degenerate on $H$, which means that $H$ is a contact structure.
\end{compactenum}
\end{definition}
The $1$-dimensional space for which \eqref{bilinear} is degenerate, which we will denote by $\ker(\psi)$, is characterised by $\ker(\psi) = \{\theta : \theta\wedge\psi=0\}$. If $\Phi$ is a diffeomorphism of $M$, the calculation $\Phi^*\theta\wedge\Phi^*\psi = \Phi^*(\theta\wedge\psi)=0$ shows $\ker(\Phi^*\psi)=\Phi^*\ker(\psi)$. This means that condition (1) is preserved by the action of the diffeomorphism group, and it is clear that conditions (2) and (3) are preserved as well. Thus the diffeomorphism group acts on the set of strongly pseudoconvex $3$-forms. Furthermore, if $\psi$ determines the contact distribution $H$, then $\Phi^*\psi$ determines the contact distribution $(\Phi^{-1})_*H$.

The decomposition $\psi=\theta\wedge\alpha$ is not unique as we can change $\theta$ to $f\theta$ and $\alpha$ to $f^{-1}\alpha$, where $f$ is a nowhere vanishing function on $M$, and add $\theta\wedge\chi$ to $\alpha$, where $\chi$ is any $1$-form.
Conditions (2) and (3) make sense independent of these choices.
In the following we choose preferred forms $\theta$ and $\alpha$.
First of all, fix the sign of $\theta$ such that $\theta\wedge\alpha^2$ is a positive form with respect to the orientation of $M$. Then scale as above by a positive function $f$ such that $\theta\wedge\alpha^2=\theta\wedge(d\theta)^2$. This fixes the contact form $\theta$ and hence a Reeb vector field $v$ such that $v$ spans $\ker d\theta$ and $\theta(v)=1$. 
Write $\omega:=d\theta$.  We have a splitting $TM= \mathbb{R}v\oplus H$,
which furthermore induces a splitting
\begin{align}
\Lambda^p TM^* = \theta\wedge \Lambda^{p-1} H^* \oplus \Lambda^p H^*
\label{splitting-forms}
\end{align}
of the bundle of $p$-forms on $M$
with the associated splitting
\begin{align}
\Omega^p = \theta\wedge\Omega^{p-1}_H \oplus \Omega^p_H,
\label{splitting-sections}
\end{align}
of $p$-forms, where we write $\Omega^p_H$ for sections of $\Lambda^p H^*$.
By the definition of $v$, we have $\omega\in\Omega^2_H$. 
We now fix $\alpha$ such that $\alpha\in\Omega^2_H$.
This does not change the previous normalisation, i.e. we have $\omega^2=\alpha^2$. Write $d\Vol_H:=\omega^2\in\Omega^4_H$ for this volume form on $H$.

Similar to the geometry of $4$-manifolds, an important role will be played by the symmetric bilinear form on $\Omega^2_H$ given by the wedge product.
We write 
\begin{align}
\sigma.\tau = \frac{\sigma\wedge\tau}{d\Vol_H},
\quad
\sigma,\tau\in\Omega^2_H.
\end{align}
Next we describe the exterior derivative $d$ under the splitting \eqref{splitting-sections}. For $X\in\Gamma(H)$ we have
$\theta([v,X])=-\omega(v,X)=0$, so that the Lie derivative $\mathcal{L}_v$ along the Reeb vector field $v$ preserves the splitting of $TM$ and \eqref{splitting-sections}. The exterior derivative on $\Omega^p_H$ is the sum of 
\begin{align}
d_H: \Omega^p_H\rightarrow \Omega^{p+1}_H,
\label{d_H}
\end{align}
and 
\begin{align}
\label{d_f}
\theta\wedge\mathcal{L}_v : \Omega^p_H\rightarrow\theta\wedge\Omega^p_H.
\end{align}
$d^2=0$ means that 
\begin{gather}
d_H^2 = -\omega\wedge\mathcal{L}_v,
\quad
\mathcal{L}_v d_H + d_H\mathcal{L}_v = 0.
\label{d2=0}
\end{gather}

Now we will use $d\psi=0$ to derive relations for $\omega$ and $\alpha$.
$d\psi=0$ is
equivalent to
\begin{align}
\omega\wedge\alpha = \theta\wedge d\alpha,
\label{rel-omega-alpha}
\end{align}
and since the left hand side is a section of $\Lambda^4 H^*$, we get $\omega\wedge\alpha=0$ and $\theta\wedge d\alpha=0$.
To sum up, $\psi$ gives orthonormal sections $\alpha,\omega$ of $\Lambda^2 H^*$ in the sense that 
\begin{align}
\omega.\omega =1,
\quad
\alpha.\alpha =1,
\quad
\omega.\alpha=0.
\label{orthonormal-eq}
\end{align}
Furthermore, we have 
\begin{align*}
d\omega=0,
\quad
d_H\alpha = 0.
\end{align*}
The orthonormal pair $(\omega,\alpha)$ defines a complex structure $K$ on $H$ with complex volume form $\omega+i\alpha$.
At each point the group of linear transformations of $TM$ preserving the structure $(\theta, \omega,\alpha)$ is isomorphic to $\SL(2,\mathbb{C})$.

\subsection{Closed 3-forms in dimension five realised by an embedding into a Calabi--Yau 3-fold}

\label{section-submanifold}

Let $Z$ be a complex manifold of complex dimension 3 with a nowhere vanishing holomorphic form $\Psi+i\hat{\Psi}$ of type $(3,0)$. The local model is $Z=\mathbb{C}^3$ with $dz^1\wedge dz^2\wedge dz^3$. Let $ M \subset Z$ be a submanifold of real dimension $5$. The pull-back of $\Psi$ to $M$ induces a closed $3$-form $\psi$ on $M$. To understand the algebraic properties of $\psi$, take $\mathbb{C}^3$ with co-ordinates $z_j=x_j+i y_j$. $\mathbb{C}^3$ has the frame
\begin{align*}
e_1=\partial_{x_1},
\quad
e_2 = \partial_{y_1},
\quad
e_3=\partial_{x_2},
\quad
e_4 = \partial_{y_2},
\quad
e_5=\partial_{x_3},
\quad
e_6 = \partial_{y_3},
\end{align*}
with dual frame $\{e^1, \dots, e^6\}$.
In this frame
\begin{align*}
\Psi
=
e^{135}-e^{146}-e^{236}-e^{245},
\quad
\hat{\Psi}
=
e^{136}+e^{145}+e^{235}-e^{246}.
\end{align*}
Given a point $p\in M$, we can always find a bi-holomorphism of $\mathbb{C}^3$ which preserves the holomorphic volume form such that 
$p=0$ and $M$ in a neighbourhood of $p$ is given as the graph
$y_3 = f(x_1,y_1,x_2,y_2,x_3)$ of a function $f$ with $f(0)=0$ and $df(0)=0$.
Then $T_p M=\SPAN\{e_1,\dots, e_5\}$
and
\begin{align*}
\psi|_p
=
e^{135}-e^{245}.
\end{align*}
At the point $p$ the skew-symmetric form \eqref{bilinear} is given by
\begin{align*}
(e_{13}-e_{24})\otimes e^{12345}.
\end{align*}
We see that this form has rank $4$, being degenerate on  the span of $\theta|_p = e^5$, and defines $H_p=\SPAN\{e_1,\dots,e_4\}$.
Thus we see that condition (1) in Definition \ref{Def-pseudoconvex}
is necessary for a closed $3$-form to be realisable as the restriction of $\Psi$ by an embedding of $M$ into $Z$.
We decompose $\psi=\theta\wedge\alpha$ and $\hat{\psi}=\theta\wedge\beta$,
where
\begin{align*}
\alpha|_p=e^{13}-e^{24},
\quad
\beta|_p = e^{14}+e^{23}.
\end{align*}
We have $\theta\wedge\alpha^2|_p=2 e^{12345}$. Thus condition (2) in Definition \ref{Def-pseudoconvex} is necessary, too, for $\psi$ to be realised by the embedding into $(Z,\Psi+i \hat{\Psi})$.

We have the relations $\alpha\wedge\beta=0$ and $\alpha^2=\beta^2$. Thus $\alpha|_H+i\beta|_H=dz_1\wedge dz_2$ induces an almost complex structure $I$ on $H$. $(H,I)$ is the real expression of the CR-structure induced by the complex structure of the ambient manifold $Z$.
The fundamental invariant of this CR-structure is its Levi form
\begin{align*}
L(X,Y) = d\theta(X,IY)-i d\theta(X,Y), \quad X,Y\in\Gamma(H).
\end{align*}
This form is definite if and only if $\theta\wedge(d\theta)^2$ is a positive multiple of $\theta\wedge\alpha^2$. Thus condition (3) in Definition \ref{Def-pseudoconvex} means precisely that given conditions (1) and (2), any embedding $\iota:M \hookrightarrow Z$ such that $\iota^*\Psi =\psi$ is a strongly pseudoconvex embedding. This motivates our chosen terminology. Strong pseudoconvexity is a helpful condition for embedding problems in CR-geometry.

We now proceed to describe the full structure induced on $M$ in the case the embedding is strongly pseudoconvex. 
Let $\omega, \alpha, \beta \in \Omega^2_H$ be the $2$-forms on $H$ obtained after the normalisation described in the previous chapter. 
Because $\hat{\psi}=\theta\wedge\beta$ is closed as well, 
analogous to \eqref{rel-omega-alpha} we get the relations
\begin{align*}
\omega\wedge\beta =0,
\quad
\theta\wedge\,d\beta = 0.
\end{align*}
The second condition is equivalent to $d_H\beta=0$. 
The triple $(\omega,\alpha,\beta)$ is orthonormal with respect to the wedge product pairing and thus defines an $\SU(2)$-structure on $M$. 
It can be thought of as a contact version of a hyperk\"ahler structure in real dimension $4$. This motivates us to make the following definition:
\begin{definition}
Let $\theta$ be a contact $1$-form on the $5$-manifold $M$ with contact distribution $H$. Suppose $\omega:=d\theta$ and a pair of $2$-forms $\alpha,\beta\in\Omega_H^2$ satisfy
\begin{align}
\omega.\omega = \alpha.\alpha = \beta.\beta =1,
\quad
\omega\wedge \alpha = \alpha\wedge\beta = \beta\wedge\omega = 0.
\end{align}
The $\SU(2)$-structure defined by $(\theta,\omega,\alpha,\beta)$ is called \textit{contact hyperk\"ahler} if it satisfies
\begin{align*}
d_H\alpha=0, \quad d_H\beta = 0.
\end{align*}
\end{definition}

$(\omega,\alpha,\beta)$ spans a positive definite subspace $\Lambda^+_H$. Together with the volume form $\Vol_H=\frac{1}{2}\omega^2$ this induces a metric $g_H$ on $H$, which is given by $g_H(X,Y)=\omega(X,IY)$.
We have a splitting 
\begin{align}
\Lambda^2 H^* = \Lambda^+_H \oplus \Lambda^-_H
\label{splitting-2-forms}
\end{align}
and write $\Omega^+_H$ and $\Omega^-_H$ for sections of $\Lambda^+_H$ and $\Lambda^-_H$, respectively.
The forms $\alpha+i \beta, \beta+ i \omega, \omega + i \alpha$ induce the almost complex structures $I, J, K$ on $H$, which satisfy the quaternionic relations and are compatible with $g_H$. Furthermore we have
\begin{align*}
\omega(X,Y) = g_H(IX,Y),
\quad
\alpha(X,Y) = g_H(JX,Y),
\quad
\beta(X,Y) = g_H(KX,Y).
\end{align*}

We can think of the realisation problem for a strongly pseudoconvex $3$-form $\psi$ as consisting of two parts. 
\begin{problem}
\label{CalabiProblem}
Find $\beta\in\Omega_H^2$ such that $(\theta,\omega,\alpha,\beta)$
forms a contact hyperk\"ahler $\SU(2)$-structure.
\end{problem}
\begin{problem}
Find an embedding for the strongly pseudoconvex CR-manifold  $(M,H,\alpha+ i \beta)$. 
\end{problem}

\subsection{Sasaki--Einstein structures and invariants of closed 3-forms in dimension five}

\label{section-Sasaki-Einstein}

Let $H$ be an oriented contact structure on a $5$-manifold $M$ with contact $1$-form $\theta$ and Reeb field $v$. Furthermore let $\omega:=d\theta, \alpha,\beta$ be an orthonormal triple on $\Lambda_H^2$. The $\SU(2)$-structure $(\theta,\omega,\alpha,\beta)$ is called \textit{Sasaki--Einstein} if $d\alpha = \theta\wedge\beta$ and $
d\beta = -\theta\wedge\alpha$, or equivalently
\begin{align}
d_H\alpha=0, \quad d_H\beta =0, \quad
\mathcal{L}_v\alpha = \beta,
\quad
\mathcal{L}_v\beta = - \alpha.
\label{SE-equations}
\end{align}
This implies that on $(0,\infty)\times M$ the conical differential forms
\begin{align*}
\Lambda
&=
(3 r^2 dr+i r^3 \theta)\wedge(\alpha-i\beta),
\\
\Omega
&=
2 r dr\wedge\theta + r^2 \omega,
\end{align*}
satisfy
\begin{align*}
\Lambda\wedge\Omega = 0, \quad \Lambda\wedge\bar{\Lambda}=2i\Omega^3,
\quad
d\Lambda=0, \quad d\Omega=0.
\end{align*}
This means that $(\Lambda,\Omega)$ defines a torsion-free $\SU(3)$-structure on the cone $(0,\infty)\times M$. In particular, the induced Riemannian cone metric is Ricci-flat and the induced metric on $M$ is Einstein.

Now we return to the structure of a strongly pseudoconvex closed $3$-form $\psi$ on the $5$-manifold $M$, which induces the $\SL(2,\mathbb{C})$-structure $(\theta,\omega,\alpha)$ on $M$ as in Section \ref{section-structure-in-dim5}. Write $\rho:=\mathcal{L}_v \alpha$. Next we show that the condition that $(\theta,\omega,\alpha,\rho)$ is a Sasaki--Einstein $\SU(2)$-structure can be expressed in terms of invariants of the $\SL(2,\mathbb{C})$-structure.

Applying the Lie derivative $\mathcal{L}_v$ to the orthonormality equations \eqref{orthonormal-eq}   gives 
\begin{align*}
\rho.\alpha = 0, 
\quad
\rho.\omega = 0.
\end{align*}
So $\rho$ is a section of the rank $4$ bundle $\Lambda^{1,1}_K \subset \Lambda^2 H^*$. The tensor $\rho$ up to an action of $\SL(2,\mathbb{C})$ is an invariant of the $2$-jet of the structure $\psi$ at a given point.
It has the scalar invariant $Q=\rho.\rho$.

If $Q=1$, then $(\omega, \alpha,\rho)$ is an orthonormal triple
and thus
$(\theta,\omega,\alpha,\rho)$ is an $\SU(2)$-structure
such that $(\omega,\alpha,\rho)$ spans the bundle $\Lambda^+_H \subset\Lambda^2 H^*$ of self-dual $2$-forms. 

We can also consider a third order invariant $\mathcal{L}_v \rho$. When $Q=1$, the orthogonality conditions  imply that 
\begin{align*}
\mathcal{L}_v\rho.\omega,
\quad
\mathcal{L}_v\rho.\rho = 0,
\quad
\mathcal{L}_v\rho.\alpha = -1.
\end{align*}
Thus $\mathcal{L}_v \rho = -\alpha\mod \Omega^-_H$.
Therefore, the conditions $Q=1$ and $\mathcal{L}_v\rho.\mathcal{L}_v\rho=1$ imply that $\alpha$ and $\rho$ solve the equations \eqref{SE-equations} and thus characterise Sasaki--Einstein structures in the setting of a strongly pseudoconvex $3$-form on a $5$-manifold.

\subsection{Example: the Minkowski problem}

\label{section-Minkowski}

Let $Z= \bC^{3}/i \bZ^{3}$ with holomorphic $3$-form induced by $i dz_{1}dz_{2} dz_{3}$ on $\bC^{3}$. Thus in standard co-ordinates $z_{a}= x_{a}+ i y_{a}$ the real $3$-form is
$$ \Psi= dy_{1} dy_{2}dy_{3} - \sum dy_{a} dx_{b} dx_{c}, $$
where $(abc)$ run over cyclic permutations. In the quotient we divide by integral translations in the $y$ coordinates. 
Let $M= \Sigma\times \bR^{3}/\bZ^{3}$ where $\Sigma$ is a compact, connected, oriented surface and consider a $3$-form of the shape
$$\psi= dt_{1}dt_{2} dt_{3} - \sum \lambda_{a} dt_{a}$$
where $\lambda_{1}, \lambda_{2}, \lambda_{3}$ are $2$-forms on $\Sigma$. We consider maps $F:M\rightarrow Z$ of the form $F(p,t)= f(p)+i t $, where
$f:\Sigma\rightarrow \bR^{3}$. Our problem is  to find $f$ such that 
\begin{equation}  f^{*}(dx_{b} dx_{c}) = \lambda_{a}, \end{equation}
for $(abc)$ cyclic.
Certainly a necessary condition is that 
\begin{equation} \int_{\Sigma} \lambda_{a}= 0  \end{equation}

The problem is invariant under special-affine transformations of $\bR^{3}$ but it is convenient, for exposition,  to use the standard Euclidean metric and  unit sphere $S^{2}\subset \bR^{3}$.
Clearly the condition that $\psi$ has maximal rank is equivalent to the condition that the $2$-forms $\lambda_{a}$ do not simultaneously vanish at any point. Thus we can define a map $L:\Sigma \rightarrow S^{2}$ and a $2$-form $\Omega$ on $\Sigma$ such that $\Omega$ is positive with respect to the orientation of $\Sigma$ and the $\bR^{3}$-valued $2$-form $\underline{\lambda}=(\lambda_{1}, \lambda_{2}, \lambda_{3})$ can be written
$ \underline{\lambda}=  L \Omega$.

We claim that the condition that $\psi$ is a strongly pseudoconvex form on $M$ is equivalent to the condition that
$L$ is an oriented local diffeomorphism. To see this, at a given point $p$ in $\Sigma$ we can assume (by rotating axes) that $\lambda_{2}, \lambda_{3}$ vanish and $\lambda_{1}$ is non-zero. Take oriented local co-ordinates $(u,v)$ on $\Sigma$ near $p$ so that $\Omega= du dv$.  Thus
$  \lambda_{a}= L_{a} du dv $. 
Then, near $p$,
$$\psi= (\sum L_{a} dt_{a}) \left( L_{1}^{-1} dt_{2}dt_{3}- du dv\right), $$
so $\theta_{0}= \sum L_{a} dt_{a}$ defines the subbundle $H$ and $\psi=\theta_{0}\wedge \alpha_{0}$ with
$\alpha_{0}= L_{1}^{-1} dt_{2} dt_{3} - dudv$. At the point $p$ the restriction of $d\theta_{0}$ to $H$ is
$$  \omega_{0}\vert_{H} = \frac{\partial L_{2}}{\partial u} du dt_{2} + \frac{\partial L_{3}}{\partial u} du dt_{3} + \frac{\partial L_{2}}{\partial v} dvdt_{2} +\frac{\partial L_{3}}{\partial v} dv dt_{3}$$ so 
$\omega_{0}\vert_{H}^{2}= -J du dv dt_{2} dt_{3}$ where
$$J=
 \frac{\partial L_{2}}{\partial u}\frac{\partial L_{3}}{\partial v} - \frac{\partial L_{2}}{\partial v}\frac{\partial L_{3}}{\partial u}, $$
 while $ \alpha_{0}\vert_{H}^{2}= -  2(du dv dt_{2} dt_{3}) $,
since $L_{1}=1$. The claim now follows because $J$ is the determinant of the derivative of $L$ at $p$ with respect to the area forms $du dv$ on $\Sigma$ and the standard area form $dA_{S^{2}}$ on $S^{2}$.

 For a map $f:\Sigma\rightarrow \bR^{3}$ the condition (1) is equivalent to the three statements:
\begin{enumerate}
\item $f$ is an immersion, so the image an immersed surface $X\subset \bR^{3}$.
\item the oriented normal to $X$ at $f(p)$ is $L(p)$.
\item the pull-back by $f$  of the (oriented) area form $dA_{X}$ on $X$ at $p$ is $\Omega(p)$.
\end{enumerate}

Since $S^{2}$ is simply connected the local diffeomorphism $L$ from $\Sigma$ to $S^{2}$ is a global diffeomorphism (so to have a pseudoconvex form $\psi$ of this shape we must suppose that $\Sigma$ is diffeomorphic to $S^{2}$). Define a positive function $K$ on $S^{2}$ by
$$  \Omega= L^{*}(K^{-1} dA_{S^{2}}). $$
Thus the function $K$ on $S^{2}$ is determined by the original data $\underline{\lambda}$.
Let $g= f \circ L^{-1}: S^{2}\rightarrow X\subset \bR^{3}$. The above statements about $f$  are equivalent to the statements that for all $\nu\in S^{2}$ the normal to $X$ at $g(\nu)$ is $\nu$ and the Gauss curvature of $X$ at $g(\nu)$ is $K(\nu)$. This is the usual formulation of the Minkowski problem for the map $g$, with prescribed Gauss curvature $K$ as a function of the normal direction. The solution of the Minkowski problem tells us that for pseusdoconvex data $\lambda_{a}$ satisfying the obvious conditions (2) there is a  solution $f$, unique up to translations.

\section{Linear Analysis}
\label{section_linear_analysis}

In this chapter we describe the linear analysis on a closed $5$-manifold which carries an oriented contact structure $H\subset TM$ with contact $1$-form $\theta$ and Reeb vector field $v$, and a Euclidean metric $g_H$ on $H$, which we extend to a Riemannian metric $g=\theta^2+g_H$ on $M$. 
Set $\omega:=d\theta$. In particular, we can apply this theory to contact hyperk\"ahler structures.

We start by describing the main differential operators. As in \eqref{d_H}
denote by $d_H: \Omega^p_H \rightarrow \Omega^{p+1}_H$ the projection of the exterior derivative to $H$. Denote by $d_H^*$ the $L^2$-adjoint of $d_H$ with respect to the metric $g$. 
$d_H^*$ explicitly is given by a formula analogous to the Riemannian setting.
\begin{lemma}
\label{Lemma-adjoint}
The adjoint of $d_H$ with respect to $g$ is given by $d^*_H=- *d_H *$, where $*$ denotes the Hodge star operator on $H$.
\end{lemma}
\begin{proof}
For $\eta\in\Omega^{k-1}_H$ and $\zeta\in\Omega^{k}_H$ we have
\begin{align*}
(\eta,d_H^* \zeta)
&=
( d_H\eta, \zeta )
=
\int_M \langle d_H\eta,\zeta\rangle \Vol
=
\int_M d_H\eta\wedge*\zeta\wedge\theta
=
\int_M d\eta\wedge *\zeta\wedge\theta
\\
&=
(-1)^{k} \int_M \eta\wedge d*\zeta\wedge\theta
=
(-1)^{k} \int_M \eta\wedge d_H * \zeta\wedge\theta
\\
&=
- \int_M \eta\wedge* (*d_H*)\zeta\wedge\theta
\\
&=
( \eta, - *d_H* \zeta ).
\end{align*}
\end{proof}
This allows us to define a Laplacian
\begin{align*}
\Delta_H:= d_H d_H^* + d_H^*d_H:
\Omega^p_H \rightarrow \Omega^p_H.
\end{align*}
As in \eqref{splitting-2-forms} we have a splitting $\Lambda^2_H=\Lambda^+_H\oplus\Lambda^-_H$ with respect to $g_H$.
The main operator in this article is the derivative
\begin{align*}
d_H^-:= \frac{1}{2}(d_H-*d_H): \Omega^1_H \rightarrow \Omega^-_H,
\end{align*}
the projection of $d_H:\Omega^1_H\rightarrow \Omega^2_H$ to anti-self-dual forms. By Lemma \ref{Lemma-adjoint} its $L^2$-adjoint is given by
\begin{align}
(d_H^-)^*=\frac{1}{2}(d_H-*d_H)^*=d_H^*.
\end{align}
To solve the linearisation of the embedding problem, we need to solve an equation of the form
\begin{align}
d_H^-\eta=\sigma
\label{main_equ}
\end{align}
for a given right-hand side $\sigma\in\Omega^-_H$. 
We will study this equation by considering the Laplacian
\begin{align*}
\Box_H:= d_H^- d_H^*: \Omega^-_H \rightarrow \Omega^-_H,
\end{align*}
which equals $\frac{1}{2}\Delta_H$ restricted to $\Omega^-_H$.

\subsection{Adapted connections}

To work with the operators introduced above, it will be useful to choose a metric connection which preserves $H$. The Levi--Civita connection $\nabla^{\mathrm{LC}}$ does not preserve $H$, so choosing such a connection comes at the cost of introducing torsion. We wish to use a connection which ``looks'' torsion-free on $H$, i.e. the torsion tensor does not have a component in $\Lambda_H^2\otimes H$.

\begin{lemma}
\label{lemma-connection}
There exists a connection $\nabla$ on $TM$ with the following properties: 
\begin{compactenum}[(a)]
\item $\nabla$ is metric, i.e. $\nabla g = 0$,
\item $\nabla$ preserves $H$, i.e. $\nabla_X Y \in\Gamma(H)$ for any $X\in\Gamma(TM)$ and $Y\in\Gamma(H)$,
\item the torsion tensor $T$ of $\nabla$ has no component in $\Lambda_H^2\otimes H$,
\item $\nabla v =0$.
\end{compactenum}
Moreover, $\nabla$ is unique up to a skew-symmetric endomorphism of $H$ and satisfies
\begin{compactenum}[(i)]
\item
If $X,Y\in\Gamma(H)$ and $\gamma \in \Omega^p_H$, then
\begin{align*}
\nabla_X Y = \pi_H(\nabla^{\mathrm{LC}}_X Y),
\quad\quad
\nabla_X \gamma = \nabla^{\mathrm{LC}}_X \gamma|_H.
\end{align*}
\item
If $e_1, \dots, e_4$ is a local orthonormal frame for $H$ with dual orthonormal co-frame $e^1, \dots, e^4$, then
\begin{align*}
d_H = \sum_{i=1}^4 e^i \wedge\nabla_i, \quad\quad
d_H^* = - \sum_{i=1}^4 e_i \lrcorner \nabla_i,
\end{align*}
where we write $\nabla_i$ for $\nabla_{e_i}$.
\item
Define $B$ by 
\begin{align*}
\nabla^{LC}_X Y = \nabla_X Y + B(X,Y) v, \quad X,Y\in\Gamma(H).
\end{align*}
Then 
\begin{align*}
B(X,Y)=
-\frac{1}{2}(\mathcal{L}_v g)(X,Y)-\frac{1}{2}\omega(X,Y).
\end{align*}
\item 
For $X,Y\in\Gamma(H)$ we have
\begin{align*}
T(X,Y)=\omega(X,Y)v.
\end{align*}
\item
Write $T_v := T(v,\cdot)$ for the contraction of the torsion tensor with $v$. Then $T_v\in\mathrm{End}(H)$. Furthermore, if we decompose $T_v=T_v^s+T_v^a$, where $T_v^s$ is symmetric and $T_v^a$ is skew-symmetric with respect to $g$, then 
\begin{align*}
g(T_v^s X,Y)=\frac{1}{2}\mathcal{L}_vg(X,Y)
\end{align*}
for all $X,Y\in\Gamma(H)$.
\end{compactenum}
\end{lemma}
\begin{proof}
We first assume that $\nabla$ exists and derive the properties (i)-(v) from the properties (a)-(d).
\\
\textbf{(i):}
If $X,Y,Z\in\Gamma(H)$, by (c) we have 
\begin{align*}
g(\nabla_X Y - \nabla_Y X, Z)
&=
g([X,Y],Z)+g(T(X,Y),Z)
\\
&=
g([X,Y],Z).
\end{align*}
Thus we can use $\nabla g =0$ as in the standard Riemannian setting to get the Koszul formula
\begin{align*}
2g(\nabla_X Y,Z)
=&
Xg(Y,Z)+Yg(Z,X)-Zg(X,Y)
\\
&+g([X,Y],Z)-g([X,Z],Y)-g([Y,Z],X)
\\
=&
2g(\nabla^{\mathrm{LC}}_X Y,Z).
\end{align*}
Because $\nabla$ preserves $H$, this means $\nabla_X Y = \pi_H(\nabla^{\mathrm{LC}}_X Y)$.
For $\gamma\in\Omega_H^p$ and $Y_1, \dots, Y_p\in\Gamma(H)$ we have
\begin{align*}
(\nabla_X \gamma)(Y_1, \dots, Y_p)
&=
X(\gamma(Y_1,\dots, Y_p))
-\sum_{j=1}^p
\gamma(Y_1, \dots, \nabla_X Y_j, \dots, Y_p)
\\
&=
X(\gamma(Y_1,\dots, Y_p))
-\sum_{j=1}^p
\gamma(Y_1, \dots, \nabla^{\mathrm{LC}}_X Y_j, \dots, Y_p)
\\
&=
(\nabla^{\mathrm{LC}}_X \gamma)(Y_1, \dots, Y_p).
\end{align*}
\textbf{(ii):}
The usual formula for the exterior derivative in terms of the torsion-free connection $\nabla^{\mathrm{LC}}$ is
\begin{align*}
d = \sum_{i=1}^4 e^i \wedge \nabla^{LC}_i + \theta\wedge\nabla^{LC}_v.
\end{align*}
By (i) we get
\begin{align*}
d_H = \sum_{i=1}^4 e^i \wedge \nabla^{LC}_i|_H
=
\sum_{i=1}^4 e^i \wedge \nabla_i.
\end{align*}
Because $\nabla$ is metric, the above formula for $d_H$ 
implies the formula for $d_H^*$ in the usual way.
\textbf{(iii):}
From (i) we have
\begin{align*}
\nabla^{LC}_X Y=\nabla_X Y + \theta(\nabla^{LC}_X Y)v.
\end{align*}
The standard Koszul formula for the Levi--Civita connection gives
\begin{align*}
2B(X,Y)=
2\theta(\nabla^{LC}_X Y)
=
-vg(X,Y)+g([X,Y],v)-g([Y,v],X)-g([X,v],Y)
\\
=
-(\mathcal{L}_v g)(X,Y)+\theta([X,Y])
=
-(\mathcal{L}_v g)(X,Y)-\omega(X,Y).
\end{align*}
\textbf{(iv):}
By (iii) the torsion is
\begin{gather*}
T(X,Y)
=
\nabla_X Y-\nabla_Y X-[X,Y]
\\
=
\nabla^{LC}_X Y-\nabla^{LC}_Y X-[X,Y]+\frac{1}{2}((\mathcal{L}_v g)(X,Y)+\omega(X,Y))v-\frac{1}{2}((\mathcal{L}_v g)(Y,X)+\omega(Y,X))v
\\
=\omega(X,Y)v.
\end{gather*}
\textbf{(v):}
Because $\nabla v=0$, for $X\in\Gamma(H)$ we have
\begin{align}
\nabla_v X = \nabla_X v + [v,X] + T(v,X)
=
\mathcal{L}_v X + T_v X,
\label{covariant_v_derivative}
\end{align}
Because $\nabla$ and $\mathcal{L}_v$ preserve $H$, we get $T_v\in\mathrm{End}(H)$. $\nabla g =0$ and formula \eqref{covariant_v_derivative} imply for $X,Y\in\Gamma(H)$
\begin{align*}
0=\nabla_v g(X,Y)=\mathcal{L}_v g(X,Y) -g(T_v X,Y)-g(X,T_v Y)
=
\mathcal{L}_v g(X,Y)-2 g(T_v^s X,Y).
\end{align*}
We now come to the existence of $\nabla$. By (a)-(d) and (i)-(v) we need to define $\nabla$ as
\begin{align*}
\nabla_X Y = \pi_H(\nabla^{\mathrm{LC}}_X Y), \quad
\nabla v = 0, 
\quad g(\nabla_v X,Y) = g(\mathcal{L}_v X,Y) + \frac{1}{2}\mathcal{L}_v g(X,Y) + g(T_v^a X,Y),
\end{align*}
where $X,Y\in\Gamma(H)$.
It is clear that this defines a connection which satisfies properties (a)-(d) and that the only freedom in the construction is the choice of $T_v^a$.
\end{proof}

A connection as in Lemma \ref{lemma-connection} allows us to compute a Weitzenb\"ock formula for $\Delta_H$. The choice of $T_v^a$ only influences the curvature term.

\begin{lemma}
\label{lemma-general_Weitzenbock}
Let $e_1, \dots, e_4$ be a local orthonormal frame for $H$ with dual co-frame $e^1, \dots, e^4$. For $\gamma\in\Gamma(\Lambda^{\bullet}_H)$, denote by $\nabla_H\gamma:= \sum_{i=1}^4 e^i\otimes \nabla_{e_i}\gamma$ the covariant derivative in the ``$H$-direction''.
Denote by $\varepsilon^k$ the wedge product with $e^k$ and by $\iota^k$ the contraction with $e_k$. Then for $\Delta_H$ we have the Weitzenb\"ock formula
\begin{align}
\label{general_Bochner_formula}
\Delta_H
=
\nabla_H^*\nabla_H 
+ \sum_{k,l} \omega_{kl}\, \varepsilon^k\iota^l \nabla_v
+ \sum_{k,l,m,n} \varepsilon^k \iota^l \varepsilon^m \iota^n R^{\nabla}_{klmn},
\end{align}
where $R^{\nabla}$ denotes the curvature tensor of the connection $\nabla$.
\end{lemma}
\begin{proof}
For a given $p\in M$, we can choose the frame $\{e_i\}$ such that $\nabla e_i|_{p} =0, i=1,\dots, 4$. This implies $\nabla_k(\iota^l \gamma)|_p = \iota^l \nabla_k \gamma|_p$ and $\nabla_k \nabla_l \gamma|_p = \nabla^2_{k,l} \gamma|_p$. Lemma \ref{lemma-connection} and the Ricci formula give
\begin{align*}
d_H d_H^* \gamma|_p
&=
-\sum_{k,l=1}^4 \varepsilon^k \nabla_k( \iota^l \nabla_l \gamma)|_p
=
-\sum_{k,l=1}^4 \varepsilon^k \iota^l \nabla_k\nabla_l\gamma|_p,
\\
d_H^* d_H \gamma|_p
&=
- \sum_{k,l=1}^4
\iota^k \nabla_k (\varepsilon^l\nabla_l \gamma)|_p
=
- \sum_{k,l=1}^4
\iota^k \varepsilon^l \nabla_k \nabla_l \gamma|_p,
\\
\Delta\gamma|_p
&=
- \sum_{k,l=1}^4
(\varepsilon^k \iota^l+\iota^k\varepsilon^l) \nabla^2_{k,l}\gamma|_p
=
-\sum_{k=1}^4 \nabla_k \nabla_k \gamma|_p
- \sum_{k < l}
(\varepsilon^k \iota^l+\iota^k\varepsilon^l) (\nabla^2_{k,l}\gamma|_p - \nabla^2_{l,k}\gamma|_p)
\\
&=
\nabla_H^*\nabla_H\gamma|_p 
- 
\sum_{k < l}
(\varepsilon^k \iota^l+\iota^k\varepsilon^l)
(R^{\nabla}(e_k,e_l)_*\gamma|_p - \nabla_{T(e_k,e_l)}\gamma|_p)
\\
&=
\nabla_H^*\nabla_H\gamma|_p 
- 
\sum_{k < l}
(\varepsilon^k \iota^l+\iota^k\varepsilon^l)
(R^{\nabla}(e_k,e_l)_*\gamma|_p - \omega(e_k,e_l) \nabla_{v}\gamma|_p)
\\
&=
\nabla_H^*\nabla_H\gamma|_p 
+\sum_{k,l=1}^4
\omega(e_k,e_l)
\varepsilon^k \iota^l  \nabla_{v}\gamma|_p
- 
\sum_{k,l=1}^4
\varepsilon^k \iota^l
R^{\nabla}(e_k,e_l)_*\gamma|_p.
\end{align*}
This gives the desired formula.
\end{proof}

\subsection{Sub-ellipticity}
\label{section-subelliptic}

In this section we describe the analytic properties of the operators $d_H^-$ and $\Box$. 
We adopt the convention from \cite{FollandKohn} that
\begin{align*}
``f(x) \lesssim g(x)\text{''}
\quad
\text{means}
\quad
``\exists C > 0\, \text{such that}\, f(x) \leq C g(x) \forall x\text{''}.
\end{align*}

By choosing an atlas for $M$ and a subordinate partition of unity
we can define as usual Sobolev spaces $L^2_s$ for sections of $TM$ and its associated bundles. We denote the norm of $L^2_s$ by $\|\cdot \|_s$. 
Here $s$ is the number of derivatives if it is an integer, but we also need to consider non-integral $s$. For details on fractional Sobolev spaces we refer to \cite[Appendix 1. and 2.]{FollandKohn}.
We also use $\mathcal{C}^k$-norms which we denote by $[[\,\cdot\, ]]_k$. Here $k$ is an integer. The splitting \eqref{splitting-forms} allows us to consider $\Lambda^{\bullet}_H$ and $\Lambda^{\pm}_H$ as subbundles of $\Lambda^{\bullet} T^*M$, and thus we also have norms for sections of these bundles.

Introduce the bilinear form
\begin{align}
\label{Q}
Q(\sigma,\sigma) = ((\Box_H + \mathbbm{1}) \sigma, \sigma)
= (d_H^-d_H^* \sigma,\sigma) +(\sigma,\sigma)
= \|d_H^*\sigma\|^2 + \|\sigma\|^2.  
\end{align}
Because $d_H^*=*d_H$ on $\Gamma(\Lambda_H^-)$ and $*$ acts isometrically, we also have the identity
\begin{align}
\label{d-identity-Q}
Q(\sigma,\sigma) = \|d_H \sigma\|^2 + \|\sigma\|^2.
\end{align}
The operator \eqref{general_Bochner_formula} is not elliptic as it does not see second derivatives in the direction of $v$. In particular, the bilinear form $Q$ is not coercive, i.e. there is \textit{no} estimate of the form
\begin{align*}
\|\sigma\|^2_1 \lesssim Q(\sigma,\sigma).
\end{align*}
However, because $H$ is a contact structure, the Reeb vector field $v$ locally can be written as a commutator of sections of $H$. In other words, a local frame of $H$ satisfies the H\"ormander condition. This leads to the fundamental ``1/2-estimate'' for the ``rough Laplacian'' $\nabla_H^*\nabla$. 

\begin{proposition}
\label{Lemma-1/2-estimate}
For $\phi\in\Omega^{\bullet}_H$ we have 
\begin{align}
\|\phi\|_{\frac{1}{2}}^2 
\lesssim
(\nabla^*_H\nabla_H \phi,\phi)+\|\phi\|^2
.
\label{general-12-estimate}
\end{align}
\end{proposition}
\begin{proof}
Because $H$ is a contact structure, a local frame $e_1, \dots, e_4$ for $H$ satisfies the H\"ormander condition.
Working locally in coordinate charts, we can therefore apply the $1/2$-estimate for functions on $\mathbb{R}^5$ \cite[Theorem 5.4.7]{FollandKohn} to each component to obtain the result.
\end{proof}
Analysing the Weitzenb\"ock formula \eqref{general_Bochner_formula} now shows that $Q$ also satisfies a sub-elliptic estimate.

\begin{proposition}
\label{Proposition-1/2-estimate-Q}
For all $\sigma\in\Omega^-_H$ we have the sub-elliptic estimate
\begin{align}
\|\sigma\|_{\frac{1}{2}}^2 \lesssim  Q(\sigma,\sigma).
\label{1/2-estimate-Q}
\end{align}
\end{proposition}
\begin{proof}
Because $\nabla$ is metric, $\nabla_v$ preserves $\Omega_H^-$. 
The action of $\omega$ on $\Lambda_H^-$ vanishes. Thus the first order term in \eqref{general_Bochner_formula} drops out and $\Box_H$ differs from $\frac{1}{2}\nabla_H^*\nabla_H$ only by the curvature term, which is an algebraic operator. Therefore, \eqref{1/2-estimate-Q} follows from \eqref{general-12-estimate}.
\end{proof}
\eqref{1/2-estimate-Q} and the Cauchy--Schwarz inequality imply
\begin{align}
\|\sigma\|_{\frac{1}{2}}^2 \lesssim \|\Box_H \sigma\|^2+\|\sigma\|^2.
\label{k=0}
\end{align}
From this Kohn--Nirenberg \cite[Lemma 3.1]{KohnNirenberg} derive higher order estimates:
we obtain for every $k\in\mathbb{N}$ and $\sigma\in\Omega^-_H$
\begin{align}
\|\sigma\|_{k+\frac{1}{2}}^2 
\lesssim 
\|\Box_H\sigma\|_{k}^2+\|\sigma\|^2
.
\label{higher-order-apriori-estimate}
\end{align}
An adapted proof can be found in \cite[section 3.6]{Hamilton1977Deformation2}.
We give a more detailed outline of Hamilton's proof in the next section when we discuss uniform estimates.

After establishing the higher order estimates, the method of elliptic regularization and standard arguments from functional analysis show that $\Box_H$ behaves like the standard Laplace-operator on a Riemannian manifold \cite[Theorem 4 (ii), (6.3)]{KohnNirenberg}: 
\begin{itemize}
\item 
$\Box_H$ is hypo-elliptic, i.e. if $\zeta$ is a distributional solution to the equation $\Box_H\zeta=\sigma$, where $\sigma$ is a smooth section of $\Lambda_H^-$, then $\zeta$ is smooth.
\item
$\ker \Box_H$ is finite-dimensional and there is an $L^2$-orthogonal
``Hodge''  decomposition
\begin{align}
\Omega_H^-
=
\ker \Box_H \oplus \mathrm{im}\, \Box_H.
\label{Hodge_ASD}
\end{align}
\item
The vanishing of $\ker \Box_{H}$ is an open condition. The estimate \eqref{k=0} depends only on finitely many derivatives of the contact structure and metric so it holds with a uniform constant in a neighbourhood of $\theta$ and $g$ which is open in the Fr{\'e}chet topology. Suppose we have a sequence $(\theta_{i}, g_{i})$ converging in that topology to $(\theta, g)$ and for each $i$ the $\Box$ operator defined by $(\theta_{i}, g_{i})$ has non-trivial kernel. We choose elements $\sigma_{i}$ of these kernels with $L^{2}$ norm $1$. Then, by
the compactness of the inclusion of $L^2_{1/2}$ in $L^{2}$,  we can suppose that these converge in $L^{2}$ to some non-zero limit $\sigma$ and the regularity statement above implies that $\sigma$ is smooth and lies in the kernel of $\Box_{H}$ for $(\theta, g)$. 
\end{itemize}
Define 
\begin{align}
\mathcal{H}:=\ker\Box_H=\{\sigma\in \Omega^-_H: d_H\sigma=0\}.
\label{obstruction-space}
\end{align}
The finite-dimensional vector space $\mathcal{H}$ is the obstruction space to solve equation \eqref{main_equ}. Decomposition \eqref{Hodge_ASD} gives

\begin{proposition}
\label{Prop-Fredholm-alternative}
Let $\sigma\in\Omega^-_H$. Then the equation 
\begin{align*}
d_H^-\eta=\sigma
\end{align*}
has a solution $\eta\in\Omega_H^1$ if and only if $\sigma\perp_{L^2} \mathcal{H}$. In particular, if $\mathcal{H}=0$, then $d_H^-$ is surjective onto $\Omega_H^-$ with right inverse $R=d_H^* \Box^{-1}_H$. 
\end{proposition}
In the case $\mathcal{H}=0$, \eqref{k=0}
gives the estimate
\begin{align*}
\|R\sigma\|_{-\frac{1}{2}} \lesssim \|\sigma\|.
\end{align*}
This is not optimal. By using pseudo-differential operators one can improve this estimate to
\begin{align}
\|R\sigma\|_{\frac{1}{2}} \lesssim \|\sigma\|.
\label{improved-estimate}
\end{align}

\subsection{Uniform estimates for the perturbed $d_H^-$-equation}

\label{section-uniform-estimates}

Any rank $3$ subbundle of $\Lambda_H^2$ which is positive definite with respect to the wedge product pairing and close to $\Lambda_H^+$ can be written as a graph 
\begin{align*}
\Lambda^+_{H,\mu}
=
\textrm{graph}(\mu)
=
\{ \zeta + \mu(\zeta)|\, \zeta \in \Lambda^+_H\}
\end{align*}
of a map
\begin{align*}
\mu: \Lambda^+_H \rightarrow \Lambda^-_H.
\end{align*}
The deformation of the bundle of anti-self-dual forms is then the graph of $\mu^*$: 
\begin{align*}
\Lambda^-_{H,\mu}=
\textrm{graph}(\mu^*)
=
\{ \zeta + \mu^*(\zeta)|\, \zeta \in \Lambda^-_H\}.
\end{align*}
Indeed, if $\sigma\in \Lambda^+_H$ and $\zeta\in\Lambda^-_H$, then
\begin{align*}
(\sigma+\mu(\sigma))\wedge(\zeta+\mu^*(\zeta)
)
=
\mu(\sigma)\wedge\zeta + \sigma\wedge \mu^*(\zeta)
=
\{
-\langle\mu(\sigma),\zeta\rangle + \langle \sigma, \mu^*(\zeta)\rangle
\}\Vol
=0.
\end{align*}
This shows $\mathrm{graph}(\mu^*)\subset (\Lambda^+_{H,\mu})^{\perp_\mu} = \Lambda^-_{H,\mu}$. Equality follows because the dimensions are equal.
The projection
$\pi_+: \Lambda^+_{H,\mu}\rightarrow \Lambda^+_H$ is an isomorphism.
Indeed, for $\zeta+\mu(\zeta)$ we have $\pi_+(\zeta+\mu(\zeta))= \zeta$.
Therefore, $\pi_+$ is injective and thus an isomorphism.
Similarly $\pi_-: \Lambda^+_{H,\mu}\rightarrow \Lambda^+_H$ is an isomorphism.
To sum up, we have bundle isomorphisms
\begin{equation}
\label{bundle-isos}
\begin{tikzcd}
\Lambda_H^+ \ar[r,"{\mathbbm{1}+\mu}",shift left] & \Lambda_{H,\mu}^+ \ar[l,"{\pi_+}",shift left],
&
\hspace{10pt}
&
\Lambda_H^- \ar[r,"{\mathbbm{1}+\mu^*}",shift left] & \Lambda_{H,\mu}^- \ar[l,"{\pi_-}",shift left].
\end{tikzcd}
\end{equation}
By keeping the volume form fixed, $\Lambda_{H,\mu}^+$ defines a new metric $\langle \cdot, \cdot \rangle_{\mu}$
on $H$ and its associated bundles, whose induced $L^2$-inner product we denote by $(\cdot, \cdot)_{\mu}$. Via the isomorphism $\mathbbm{1}+\mu^*$ we can pull-back $\langle \cdot, \cdot \rangle_{\mu}$ to a metric $\dlangle \cdot, \cdot \drangle_{\mu}$ on $\Lambda_H^-$ with induced $L^2$-inner product $\dlpara\cdot, \cdot \drpara_{\mu}$. An explicit calculation shows that on $\Lambda_H^-$ the inner products are related by $\dlangle\, \cdot\, ,\, \cdot\, \drangle_{\mu} = \langle \, \cdot\, , (\mathbbm{1}-\mu\mu^*)\,\cdot\, \rangle$. 

The decomposition $\Lambda_H^2=\Lambda_{H,\mu}^+\oplus \Lambda_{H,\mu}^-$
gives rise to a perturbation $d_{H,\mu}^-$ of $d_H^-$ and a corresponding perturbation of equation \eqref{main_equ}.
We want to view differential operators arising from the perturbed decomposition of $\Lambda_H^2$ as operators acting between sections of fixed bundles. 
The isomorphisms \eqref{bundle-isos} allow us to identify $d_{H,\mu}^-$ with the operator $D_{\mu}:= \pi_-\circ d_{H,\mu}^-: \Gamma(\Lambda^1_H)\rightarrow \Gamma(\Lambda_H^-)$. 
With respect to $(\cdot ,\cdot)_{\mu}$ we have an adjoint $(d_{H,\mu}^-)^{*_\mu}$, which by Lemma \ref{Lemma-adjoint} coincides with $d_H^{*_{\mu}}$.
Set $D_{\mu}^*:= (d_{H,\mu}^-)^{*_{\mu}} \circ (\mathbbm{1}+\mu^*) : \Gamma(\Lambda_H^-)\rightarrow \Gamma(\Lambda_H^1)$. This notation is justified as we have $\dlpara D_{\mu} \eta, \sigma\drpara_{\mu} = (\eta, D_{\mu}^* \sigma )_{\mu}$ for $\eta\in\Gamma(\Lambda_H^1)$ and $\sigma\in\Gamma(\Lambda_H^-)$. The relevant second order operator is $E_{\mu}:= D_{\mu} \circ D_{\mu}^* = \pi_- \circ d_{H,\mu}^- \circ (d_{H,\mu}^-)^{*_{\mu}} \circ (\mathbbm{1}+\mu^*)
=\pi_- \circ \Box_{H,\mu} \circ (\mathbbm{1}+\mu^*)
: \Gamma(\Lambda_H^-)\rightarrow \Gamma(\Lambda_H^-)$.

The operator $E_{\mu}$ enjoys the same analytic properties as the operator $\Box_H$. In particular $E_{\mu}$ is sub-elliptic and at the end of the previous section we have seen that $\ker\Box_H=0$ implies that  $\ker E_{\mu}$ vanishes if $\mu$ is sufficiently small. 
To apply the Nash--Moser implicit function theorem to the perturbative embedding problem, we need to carefully check how the higher order estimates \eqref{higher-order-apriori-estimate} depend on the parameter $\mu$. For example an estimate of the form
\begin{align*}
\|\sigma\|_{k+\frac{1}{2}} \lesssim
\|E_{\mu}\sigma\|_{k}+([[\mu]]_{2k}+1) \|\sigma\|_2
\end{align*}
would not be enough \cite[Counterexample I.5.5.4]{HamiltonIFT}.

We start by deriving an uniform version of the ``1/2''-estimate \eqref{1/2-estimate-Q}. Analogously to \eqref{Q} define a quadratic form
\begin{align*}
Q_{\mu}(\sigma,\sigma)
=
\dlpara (E_{\mu}+\mathbbm{1})\sigma,\sigma \drpara_{\mu}
=
( D_{\mu}^* \sigma, D_{\mu}^*\sigma )_{\mu}
+ \dlpara \sigma,\sigma \drpara_{\mu}.
\end{align*}
The analogue of formula \eqref{d-identity-Q} is 
\begin{align*}
Q_{\mu}(\sigma,\sigma)
=
(d_H (\mathbbm{1}+\mu^*)\sigma,d_H (\mathbbm{1}+\mu^*)\sigma)_{\mu}
+
(\sigma,\sigma)_{\mu}.
\end{align*}
In the following we will assume a $\mathcal{C}^0$-bound on $\mu$. This then implies that we have a uniform equivalence of $L^2$-products
\begin{align*}
(\cdot, \cdot) \lesssim (\cdot, \cdot)_{\mu} \lesssim (\cdot,\cdot).
\end{align*}

\begin{proposition}
Under a $\mathcal{C}^1$-bound on $\mu$ for all $\sigma\in\Omega_H^-$ we have 
\begin{align}
\|\sigma\|_{\frac{1}{2}}^2
\lesssim
Q_{\mu}(\sigma,\sigma).
\label{uniform-1/2-estimate}
\end{align}
\end{proposition}
\begin{proof}
We have
\begin{align*}
Q_{\mu}(\sigma,\sigma)
&=
(d_H (\mathbbm{1}+\mu^*)\sigma,d_H (\mathbbm{1}+\mu^*)\sigma)_{\mu}
+
(\sigma,\sigma)_{\mu}
\\
&\gtrsim
(d_H (\mathbbm{1}+\mu^*)\sigma,d_H (\mathbbm{1}+\mu^*)\sigma)
+
(\sigma,\sigma)
\\
&=
Q_0(\sigma,\sigma) + 2(d_H\sigma,d_H \mu^*\sigma) + \|d_H \mu^* \sigma\|^2
\\
&\geq 
Q_0(\sigma,\sigma)
-2 \|d_H\sigma\| \|d_H\mu^*\sigma\|.
\end{align*}
By Lemma \ref{lemma-connection} (ii) we have
\begin{align*}
d_H\mu^* \sigma = \sum_{i=1}^4 e^i\wedge \nabla_i(\mu^*\sigma)
=
\sum_{i=1}^4 e^i\wedge (\nabla_i\mu^*)\sigma
-
\sum_{i=1}^4 e^i\wedge \mu^*\nabla_i\sigma.
\end{align*}
Thus with Proposition \ref{Proposition-1/2-estimate-Q}
and formula \eqref{d-identity-Q}
\begin{align*}
\|d_H \mu^* \sigma \|
&\lesssim [[\mu]]_1 \|\sigma\| + [[\mu]]_{0} \|\nabla_H \sigma\|
\lesssim [[\mu]]_1 (\|\nabla_H\sigma\|+\|\sigma\|)
\\
&\lesssim [[\mu]]_1 \sqrt{Q(\sigma,\sigma)}
\lesssim
[[\mu]]_1 (\|d_H \sigma\|+\|\sigma\|).
\end{align*}
If we denote the constant in the last inequality by $C$, with \eqref{d-identity-Q} this gives
\begin{align*}
Q_{\mu}(\sigma,\sigma) 
&\gtrsim Q_0(\sigma,\sigma)-2C[[\mu]]_1 \|d_H \sigma\| (\|d_H\sigma\|+\|\sigma\|)
\\
&\gtrsim 
Q_0(\sigma,\sigma) -3C
 [[\mu]]_1 (\|d_H\sigma\|^2 + \|\sigma\|^2)
\\
&\gtrsim
Q_0(\sigma,\sigma)
- 3C [[\mu]]_1 Q_0(\sigma,\sigma).
\end{align*}
Then if $[[\mu]]_1 < \frac{1}{6C}$ we get with \eqref{1/2-estimate-Q} 
\begin{align*}
\|\sigma\|_{\frac{1}{2}}^2 \lesssim Q_0(\sigma,\sigma) \lesssim Q_{\mu}(\sigma,\sigma).
\end{align*}
\end{proof}

Next we explain how to obtain uniform higher order estimates from the uniform ``1/2-estimate''. The proof is due to Hamilton \cite{Hamilton1977Deformation2}.

\begin{proposition}
For all $\mu$ in a sufficiently small neighbourhood of $0$, there is an estimate
\begin{align}
\|\sigma\|_{k+\frac{1}{2}} \lesssim
\|E_{\mu}\sigma\|_{k}+([[\mu]]_{k+2}+1) \|\sigma\|_2
.
\label{uniform_estimate_E}
\end{align}
for each $k\in\mathbb{N}$.
\end{proposition}
\begin{proof}
We use induction on $k$. For $k=0$ the estimate follows from \eqref{uniform-1/2-estimate} and the Cauchy--Schwarz inequality:
\begin{align*}
\|\sigma\|^2_{\frac{1}{2}}
\lesssim 
(E_{\mu} \sigma,\sigma) + \|\sigma\|^2
\lesssim
\|E_{\mu}\sigma\|^2+\|\sigma\|^2.
\end{align*}
In the derivation of the higher order estimates we want to avoid dealing with mixed partial derivatives. We use the following trick from \cite[p.420]{Hamilton1977Deformation2}:
for each $k\in\mathbb{N}$ there exist $N$ (which is allowed to depend on $k$ and can be large) vector fields $X_1, \dots, X_N$ 
such that 
\begin{align*}
\|\sigma\|^2_{k+\frac{1}{2}} \leq \sum_{l=0}^k \sum_{j=1}^N \|\nabla_{X_j}^l \sigma\|^2_{\frac{1}{2}}.
\end{align*}
On $\mathbb{R}^n$ this follows from the statement from algebra that each homogeneous polynomial $q(\underline{\partial})$ of degree $k$ depending on $n$ formal variables $\underline{\partial}=(\partial_1, \dots, \partial_n)$ can be written as a linear combination of the $k$-th powers of linear polynomials $p_1(\underline{\partial}), \dots, p_N(\underline{\partial})$.
E.g. we have 
\begin{align*}
6\partial_1\partial_2^2
=
\partial_1^3-6\partial_2^3 -2 (\partial_1+\partial_2)^3+(\partial_1+2\partial_2)^3.
\end{align*} 
On the manifold $M$ we can find the vector fields $X_j, j=1, \dots, N$, by choosing a partition of unity to reduce to the local case.
In the following fix $j$ and just write $\nabla$ for $\nabla_{X_j}$.
Inserting $\nabla^k\sigma$ into \eqref{uniform-1/2-estimate} gives
\begin{align}
\|\nabla^k\sigma\|^2_{\frac{1}{2}}
\lesssim
Q_{\mu}(\nabla^k\sigma,\nabla^k\sigma)
\lesssim
|Q_{\mu}(\sigma,\nabla^{2k}\sigma)|
+
|Q_{\mu}(\nabla^k\sigma,\nabla^k\sigma)
-(-1)^k
Q_{\mu}(\sigma,\nabla^{2k}\sigma)|.
\label{higher-order-estimate}
\end{align}
Because $Q_{\mu}$ depends on $\mu$ and its first derivative, 
by the ``uniform Kohn--Nirenberg Lemma'' \cite[p.448]{Hamilton1977Deformation2} we can estimate the error term as
\begin{gather}
|Q_{\mu}(\nabla^k\sigma,\nabla^k\sigma)
-(-1)^k
Q_{\mu}(\sigma,\nabla^{2k}\sigma)|
\nonumber
\\
\lesssim
\|\sigma\|_k^2
+[[\mu]]_{k+2}^2 \|\sigma\|_1^2
\lesssim
(sc)\|\sigma\|_{k+\frac{1}{2}}^2+(lc)\|\sigma\|_{k-\frac{1}{2}}^2
+[[\mu]]_{k+2}^2 \|\sigma\|_1^2
,
\label{estimate-errorterm}
\end{gather}
where we can make the ``small'' constant (sc) arbitrarily small by choosing the ``large'' constant (lc) sufficiently large.
The adjoint $\nabla_{\mu}^*$ of $\nabla$ with respect to $\dlpara \cdot,\cdot \drpara_{\mu}$ has the form
$\nabla_{\mu}^*=-\nabla+a(\mu)$,
where $a(\mu)$ is a linear partial differential operator of degree $0$ whose coefficients depend on $\mu$ and its first derivative.
Thus with \cite[Lemma on bottom of p.448]{Hamilton1977Deformation2}
we get
\begin{align*}
|Q_{\mu}(\sigma,\nabla^{2k}\sigma)|
&=
|\dlpara(E_{\mu}+\mathbbm{1})\sigma,\nabla^{2k}\sigma\drpara_{\mu}|
=
|\dlpara{\nabla_{\mu}^*}^k (E_{\mu}+\mathbbm{1})\sigma,\nabla^k\sigma\drpara_{\mu}|
\\
&\lesssim
\|{\nabla_{\mu}^*}^k (E_{\mu}+\mathbbm{1})\sigma\|^2 + \|\sigma\|_k^2
\\
&
\lesssim
\|(E_{\mu}+\mathbbm{1})\sigma\|^2_{k}
+[[\mu]]_k^2 \|(E_{\mu}+\mathbbm{1})\sigma\|^2+\|\sigma\|_k^2
\\
&\lesssim
\|E_{\mu}\sigma\|_k^2 + [[\mu]]_k^2 \|\sigma\|^2 + [[\mu]]_k^2 \|E_\mu \sigma\|^2 + \|\sigma\|_k^2.
\end{align*}
Under a $\mathcal{C}^2$-bound on $\mu$ we have $\|E_{\mu}\sigma\|\lesssim \|\sigma\|_2$. Using this, we then get
\begin{align}
|Q_{\mu}(\sigma,\nabla^{2k}\sigma)|
\lesssim
\|E_{\mu}\sigma\|_k^2 + [[\mu]]_k^2 \|\sigma\|_2^2 + (sc)\|\sigma\|_{k+\frac{1}{2}}^2+(lc)\|\sigma\|_{k-\frac{1}{2}}^2.
\label{estimate-integrationbyparts}
\end{align}
Combining \eqref{higher-order-estimate}, \eqref{estimate-errorterm}, \eqref{estimate-integrationbyparts}
gives
\begin{align*}
\|\nabla^k\sigma\|_{\frac{1}{2}}^2
\lesssim
\|E_{\mu}\sigma\|_k^2
+ [[\mu]]_{k+2}^2 \|\sigma\|_2^2 + (sc)\|\sigma\|_{k+\frac{1}{2}}^2+(lc)\|\sigma\|_{k-\frac{1}{2}}^2.
\end{align*}
Summing over all vector fields $X_j$, rearranging and using the induction hypothesis gives the desired estimate.
\end{proof}

To derive uniform estimates for the right inverse $R_{\mu}$ of $D_{\mu}$ we will use the ``second Moser estimate'' \cite[p. 439]{Hamilton1977Deformation2}: If the operator $L(m)f$ depends on derivatives of $m$ up to order $r$, possibly in a non-linear way, and is linear and of order $s$ in $f$, then under a $\mathcal{C}^r$-bound on $m$ we have a uniform estimate
\begin{align}
\|L(m)f\|_k \lesssim \|f\|_{s+k} + [[m]]_{k+r} \|f\|_s.
\label{2nd-Moser_estimate}
\end{align}

\begin{corollary}
\label{corollary-uniform_estimate_D}
There exists $l\in\mathbb{N}$, such that for all $\mu$ in a sufficiently small neighbourhood of $0$ the right inverse $R_{\mu}:= D_{\mu}^* E_{\mu}^{-1}$ of $D_{\mu}$ satisfies an estimate
\begin{align}
\|R_{\mu}\sigma\|_{k} \lesssim \|\sigma\|_{k+1}+([[\mu]]_{k+3}+1)\|\sigma\|_{l}
\label{uniform-right-inv-estimate}
\end{align}
for every $k\in\mathbb{N}$.
\end{corollary}
\begin{proof}
With a contradiction argument \cite[Lemma on p.453]{Hamilton1977Deformation2} one can conclude from the estimates \eqref{uniform_estimate_E} that there exists $l\in\mathbb{N}$ such that for all $\mu$ in some neighbourhood of $0$ and $\sigma\in\Omega^-_H$  we have
\begin{align}
\|\sigma\|_2 \lesssim \|E_{\mu}\sigma\|_{l}.
\label{Poincare-inequality}
\end{align}
By \eqref{uniform_estimate_E} and \eqref{Poincare-inequality}
we have
\begin{align}
\|E^{-1}_{\mu}\sigma\|_{k+\frac{1}{2}}
\lesssim
\|\sigma\|_{k}+([[\mu]]_{k+2}+1) \|\sigma\|_{l}.
\label{E-1tame}
\end{align}
Applying the second Moser estimate \eqref{2nd-Moser_estimate} to the operator $D_{\mu}^*$, which depends on $\mu$ and its first derivative, gives for sufficiently small $\mu$
\begin{align*}
\|R_{\mu}\sigma\|_k
=
\|D_{\mu}^*E_{\mu}^{-1}\sigma\|_k
\lesssim
\|E_{\mu}^{-1}\sigma\|_{k+1}
+
[[\mu]]_{k+1} \|E_{\mu}^{-1}\sigma\|_1,
\end{align*}
which together with \eqref{Poincare-inequality} leads to the estimate
\begin{align*}
\|R_{\mu}\sigma\|_k
\lesssim
\|E_{\mu}^{-1}\sigma\|_{k+\frac{3}{2}}
+
[[\mu]]_{k+1} \|\sigma\|_l.
\end{align*}
With \eqref{E-1tame} we get the desired estimate \eqref{uniform-right-inv-estimate}.
\end{proof}

\subsection{The obstruction space and the tangential Cauchy--Riemann operator}

Here we briefly review the $\bar{\partial}_b$-operator associated with the CR-structure $I$, i.e. the one induced on $M$ from the complex structure of the ambient manifold $Z$. 
For details we refer to \cite{Boggess}.
The action of $I$ on $H\otimes \mathbb{C}$ has eigenvalues $\pm i$ with eigenbundle decomposition $H\otimes \mathbb{C}=H^{1,0}\oplus H^{0,1}$, where $H^{0,1}=\overline{H^{1,0}}$.
This has a dual decomposition $H^*\otimes \mathbb{C}= (H^*)^{1,0}\oplus (H^*)^{0,1}$. Set $\Lambda^{p,q}_H = \Lambda^p (H^*)^{1,0} \otimes \Lambda^q (H^*)^{0,1}$ and $\Lambda^{p,q}_M = \theta\wedge \Lambda^{p-1,q}_H \oplus \Lambda^{p,q}_H$. 
Write $\Omega^{p,q}_H$ for sections of $\Lambda^{p,q}_H$ and $\Omega^{p,q}_M$ for sections of $\Lambda^{p,q}_M$.
We have decompositions $\Lambda^k_H = \oplus_{p+q=k} \Lambda^{p,q}_H$ and 
$\Lambda^k_M = \oplus_{p+q=k} \Omega^{p,q}_M$.
Write $\pi^{p,q}_H$ for the projection onto $\Lambda^{p,q}_H$ and $\pi^{p,q}_M$ for the projection onto $\Lambda^{p,q}_M$.

The $\bar{\partial}_b$-operator is given by
\begin{align*}
\bar{\partial}_b = \pi^{p,q+1}_M \circ d  : \Omega^{p,q}_M \rightarrow \Omega^{p,q+1}_M
\end{align*}
and satisfies $\bar{\partial}_b^2=0$, i.e. leads to a complex.
For $\gamma\in\Omega^{p,q}_H$, $d_H\gamma$ takes values in $\Omega^{p+1,q}_H\oplus \Omega^{p,q+1}_H$, i.e. we have a splitting $d_H = \partial_H+ \bar{\partial}_H$.  

Under the identification $\Omega^{p,q}_M\cong \Omega^{p-1,q}_H\oplus \Omega^{p,q}_H$, in matrix notation $\bar{\partial}_b$ is given by
\begin{align*}
\bar{\partial}_b
=
\begin{blockarray}{ccc}
\Omega^{p-1,q}_H & \Omega^{p,q}_H
\\
\begin{block}{(cc)c}
-\bar{\partial}_H &  S & \Omega^{p-1,q+1}_H
\\
\omega\wedge\cdot & \bar{\partial}_H & \Omega^{p,q+1}_H
\\
\end{block}
\end{blockarray}
\end{align*}
where we write $S:= \pi^{p-1,q+1}_H\circ \mathcal{L}_v$.
The adjoint $\bar{\partial}_b^*:\Omega^{p,q+1}_M\rightarrow\Omega^{p,q}_M$ of $\bar{\partial}_b:\Omega^{p,q}_M\rightarrow\Omega^{p,q+1}_M$ is given by
\begin{align*}
\bar{\partial}_b^*
=
\begin{blockarray}{ccc}
\Omega^{p-1,q+1}_H & \Omega^{p,q+1}_H
\\
\begin{block}{(cc)c}
-\bar{\partial}_H^* &  \Lambda & \Omega^{p-1,q}_H
\\
S^* & \bar{\partial}_H^* & \Omega^{p,q}_H
\\
\end{block}
\end{blockarray}
\end{align*}
where $\Lambda$ is the adjoint of $\omega\wedge\cdot$. 
A similar computation as in the proof of Lemma \ref{Lemma-adjoint} shows that $\bar{\partial}_H^* = \pm * \partial_H *$.
The decomposition $\Lambda^2_H=\Lambda^+_H\oplus \Lambda^-_H$
relates to the decomposition in $(p,q)$-forms as
\begin{align}
\Lambda^{1,1}_H = \mathbb{C}\omega \oplus \Lambda_H^-\otimes\mathbb{C},
\quad
\Lambda^{2,0}_H=\mathbb{C}(\alpha+i\beta).
\label{11vs-}
\end{align} 

\begin{lemma}
\label{lemma-coclosed}
If $\sigma\in\mathcal{H}$, then $\bar{\partial}_b^*\sigma=0$.
\end{lemma}
\begin{proof}
By the decomposition \eqref{11vs-} the map $\Lambda$ vanishes on $\Lambda^-_H$. Therefore, for $\sigma\in\mathcal{H}$
we have
\begin{align*}
\Lambda\sigma=0,
\quad\quad
\bar{\partial}_H^*\sigma = \pm *\partial_H *\sigma
= \pm * \partial_H \sigma
=
\pm * \pi^{2,1}_H d_H\sigma
=0
.
\end{align*}
This implies $\bar{\partial}^*_b\sigma=0$.
\end{proof}
If the ambient space $Z$ is a Stein manifold, the cohomology groups of the $\bar{\partial}_b$-complex vanish and there are no non-trivial $\bar{\partial}_b$-harmonic forms on $M$  (see Appendix \ref{appendix-vanishing}).
This has two consequences. First, every $\sigma\in\mathcal{H}$ is $\bar{\partial}_b^*$-exact. Second, if $\bar{\partial}_b\sigma=0$, then $\sigma=0$.
With this in mind, in the following we compute $\|\bar{\partial}_b\sigma\|$ for $\sigma\in\mathcal{H}$.

\begin{proposition}
\label{prop-norm-delbar}
For $\sigma\in\mathcal{H}$ we have 
\begin{align*}
|\bar{\partial}_b\sigma|^2
=
\frac{1}{4}
(\langle\sigma,\mathcal{L}_v\alpha\rangle^2 + \langle\sigma,\mathcal{L}_v\beta\rangle^2 ).
\end{align*}
In particular, if $(\theta,\omega,\alpha,\beta)$ is a Sasaki--Einstein structure, this implies with the equations \ref{SE-equations} that $\bar{\partial}_b\sigma=0$.
\end{proposition}
\begin{proof}
Because $d_H\sigma=0$, we have $\bar{\partial}_b \sigma = \theta\wedge\pi^{0,2}\mathcal{L}_v \sigma$.
Write $\pi^{0,2}\mathcal{L}_v \sigma=u (\alpha-i\beta)$ for some complex valued function $u$ on $M$.
Then
\begin{align*}
\mathcal{L}_v \sigma\wedge(\alpha+i\beta)
=4 u \Vol_H
.
\end{align*}
Now we have
\begin{align*}
0 = \mathcal{L}_v(\sigma\wedge\alpha)
=
\mathcal{L}_v \sigma\wedge\alpha + \sigma\wedge\mathcal{L}_v\alpha.
\end{align*}
This implies $\mathcal{L}_v\sigma\wedge\alpha=-\sigma\wedge\mathcal{L}_v\alpha=\langle \sigma,\mathcal{L}_v \alpha\rangle\Vol_H$.
Combining the above gives
\begin{align*}
u = \frac{1}{4}(\langle\sigma,\mathcal{L}_v\alpha\rangle + i\langle\sigma,\mathcal{L}_v\beta\rangle )
\end{align*}
and
\begin{align*}
|\bar{\partial}_b\sigma|^2=
|\pi^{0,2}\mathcal{L}_v\sigma|^2
=
4 u \bar{u}
=
\frac{1}{4}
(\langle\sigma,\mathcal{L}_v\alpha\rangle^2 + \langle\sigma,\mathcal{L}_v\beta\rangle^2 ).
\end{align*}
\end{proof}

Lemma \ref{lemma-coclosed}, Proposition \ref{prop-norm-delbar}
and Corollary \ref{corollary-Stein-vanishing} 
prove the third item in Theorem \ref{mainTHM}.
\begin{corollary}
\label{vanishing}
If $F:M\rightarrow Z$ is a strongly pseudoconvex embedding of $M$ into a Stein manifold $Z$ such that the induced $\SU(2)$-structure is Sasaki--Einstein, then $\mathcal{H}=0$.
\end{corollary}

\section{The perturbative embedding problem}

\label{section-Nash-Moser}

In this section we will prove the second item in Theorem \ref{mainTHM}.
Denote by $\mathcal{E}(M,Z)$ the set of smooth embeddings of $M$ into $Z$.
Because $M$ is compact, this is a tame Fr{\'e}chet manifold \cite[Corollary II.2.3.2]{HamiltonIFT}. If $F:M\hookrightarrow Z$ is an embedding, a chart $u: \mathcal{U}\subset \mathcal{E}(M,Z)\rightarrow \mathcal{C}^{\infty}(M,F^*TZ)$ around $F$ is obtained as follows: Choose a Riemannian metric on $Z$ with associated exponential map, which maps a neighbourhood of the zero section in $TZ$ to $Z$. Then for $\widetilde{F}\in\mathcal{E}(M,Z)$ close to $F$, we define $u(\widetilde{F})\in\mathcal{C}^{\infty}(M,F^*TZ)$ by  $\exp(u(\widetilde{F}))=\widetilde{F}$.

Let $F\in\mathcal{E}(M,Z)$ be a strongly pseudoconvex 
embedding with induced $3$-form $\psi:= F^*\Psi$ and contact hyperk\"ahler  $\SU(2)$-structure $(\theta,\omega,\alpha,\beta)$. Denote by $\mathcal{E}_0(M,Z)$ the connected component of $F$.
To prove the existence part in Theorem \ref{mainTHM}, we want to show that the map
\begin{gather*}
P: \mathcal{E}_0(M,Z)\rightarrow \psi+d\Omega^2(M),
\\
\widetilde{F} \mapsto \widetilde{F}^*\Psi.
\end{gather*} 
is surjective onto a nighbourhood of $\psi$.
With the aim of applying a type of inverse function theorem, 
we study the properties of the derivative of $P$ at $F$, which is a map
\begin{gather*}
DP(F): T_{F}\mathcal{E}(M,Z)=\mathcal{C}^{\infty}(M,F^*TZ)
\rightarrow d\Omega^2(M).
\end{gather*}
We have the identification 
\begin{align*}
F^*TZ
=
\mathbb{R}v \oplus \mathbb{R}Iv\oplus H,
\end{align*}
where $I$ is the complex structure of the ambient space $Z$.
Consider the map
\begin{gather*}
K(F): \mathcal{C}^{\infty}(M,F^*TZ)= \mathcal{C}^{\infty}(M)\oplus \mathcal{C}^{\infty}(M)\oplus \mathcal{C}^{\infty}(M,H)
\rightarrow \Omega^2(M),
\\
(h_{\alpha}, h_{\beta},w_H)\mapsto
F^*((h_{\alpha}  v + h_{\beta} Iv + w_H)\lrcorner \Psi) 
=
h_{\alpha}\, \alpha + h_{\beta}\, \beta + (w_H\lrcorner\alpha)\wedge\theta.
\end{gather*}
Then $DP(F)=d\circ K(F)$. 
We first show that the cokernel of $DP(F)$ can be identified with the cokernel $\mathcal{H}$ of the map $d_H^-$.
We first note that for $d\sigma\in d\Omega^2(M)$, the property $\sigma\perp_{L^2} \mathcal{H}$, where we use the $L^2$-norm defined by the fibre-wise inner product $g$, is independent of the primitive $\sigma$ and can be expressed in terms of 
 the ``linking form'' 
\begin{align*}
L(d\tau,d\sigma) = \int_M d\tau\wedge\sigma
\end{align*}
on $d\Omega^2(M)$,
which is well-defined by Stokes' theorem. 
If $\zeta\in\mathcal{H}$, then $d(\theta\wedge\zeta) = \omega\wedge\zeta-\theta\wedge d_H\zeta =0$, and 
\begin{align*}
(\zeta,\sigma)
=-
\int_M \theta\wedge\zeta\wedge\sigma
=-L(\theta\wedge\zeta,d\sigma).
\end{align*}

\begin{proposition}
\label{Prop-CokernelDP}
$d\sigma\in d\Omega^2(M)$ is in the image of $DP(F)$ if and only if 
$L(\theta\wedge\zeta,d\sigma)=0$
for all $\zeta\in\mathcal{H}$.
\end{proposition}
\begin{proof}
Because $\im\, K(F) \perp \Lambda^-_H$, it is clear that every element in $\im\, DP(F)=d(\im\, K(F))$ has a primitive which is $L^2$-orthogonal to $\mathcal{H}\subset\Omega^-_H$. 

To show that the condition is sufficient, suppose that $d\sigma$ satisfies $L(\theta\wedge\zeta,d\sigma)=0$ for all $\zeta\in\mathcal{H}$. 
Decompose $\sigma$ as
\begin{align*}
\sigma=\sigma_H +\chi\wedge\theta=\sigma_+ + \sigma_-+\chi\wedge\theta, \end{align*}
where $\sigma_H\in\Omega_H^2$, $\sigma_{\pm} = \pi_{\pm}(\sigma_H)\in\Omega^{\pm}_H$ and $\chi\in\Omega_H^1$.
Because the freedom in choosing a primitive in $\Omega^2(M)$ is $d\Omega^1(M)$, we want to add an exact $2$-form to $\sigma$ to obtain an element in $\im\, K(F)$. Comparing $\im\, K(F)$ and $\Lambda^2_M$, we see that our task is to eliminate the component of $\sigma$ in $\Lambda^-_{H}\oplus \mathbb{R}\omega$.
Our assumption on $d\sigma$ is that $\sigma_-\perp_{L^2}\mathcal{H}$. 
Therefore, by Proposition \ref{Prop-Fredholm-alternative} we have a solution $\eta\in\Omega_H^1$ of the equation
\begin{align}
d_{H}^- \eta = \sigma_-.
\label{crucial-equ}
\end{align} 
Thus we have
\begin{align*}
\sigma-d\eta
= \sigma_+ - d^+_H \eta + (\chi+\mathcal{L}_v \eta)\wedge\theta
=:
\sigma'_+ + \chi'\wedge\theta.
\end{align*} 
Next decompose 
\begin{align*}
\sigma'_+ = h_{\omega} \omega + h_{\alpha} \alpha + h_{\beta} \beta.
\end{align*}
We eliminate the term involving $\omega$ by
subtracting $d(h_{\omega} \theta)$. We get
\begin{align*}
\sigma - d(\eta+h_{\omega} \theta)
=
h_{\alpha} \alpha + h_{\beta} \beta + (\chi'-d_H h_{\omega})\wedge\theta
=:
h_{\alpha} \alpha + h_{\beta} \beta + \chi''\wedge\theta
.
\end{align*}
This expression now lies in the image of $K(F)$.
\end{proof}

Proposition \ref{Prop-CokernelDP} shows that $DP(F)$ is surjective if the obstruction space $\mathcal{H}$ vanishes. 
Because the vanishing of $\mathcal{H}$ is an open condition, $DP(\widetilde{F})$ is then surjective for all embeddings $\widetilde{F}$ in an open neighbourhood of $F$ in $\mathcal{E}(M,Z)$. To apply the Nash--Moser implicit function theorem, we need to carefully check how the norm of a right inverse depends on the parameter.
This will prove the existence part in Theorem \ref{mainTHM}.

\begin{proposition}
Suppose that $\mathcal{H}=0$ 
for the $\SU(2)$-structure induced by the strongly pseudoconvex embedding $F\in\mathcal{E}(M,Z)$. Then the map $P$ is surjective onto an open neighbourhood of $\psi$.
\end{proposition}
\begin{proof}
The conditions in Definition \ref{Def-pseudoconvex} depend on two derivatives of $F$. Thus 
for any $\widetilde{F}$ which is $\mathcal{C}^2$-close to $F$ the induced 3-form $\tilde{\psi}:=P(\widetilde{F})$ is strongly pseudoconvex as well.
Denote by $(\tilde{\theta},\tilde{\omega},\tilde{\alpha},\tilde{\beta})$ the $\SU(2)$-structure induced by $\widetilde{F}$.
We observe here that this $\SU(2)$-structure depends on $\widetilde{F}$ and its first three derivatives. One derivative to determine $\tilde{\psi}$. Another derivative is needed to normalise as in section \ref{section-structure-in-dim5}. Third derivatives come in by computing $\omega=d\theta$.

To describe a right inverse for $DP(\widetilde{F})$, we will distinguish three cases:
\begin{compactenum}[(1)]
\item
$\widetilde{H}=H$ and $\tilde{\theta}=\theta$ (and thus $\tilde{v}=v$).
\item
$\widetilde{H}=H$ but $\tilde{\theta}\neq\theta$ (and thus $\tilde{v}\neq v$).
\item
The contact distributions $\widetilde{H}$ and $H$ are different.
\end{compactenum}
We start by explaining case (1). In this case the splitting \eqref{splitting-forms} remains the same. 
However, $(\omega,\tilde{\alpha},\tilde{\beta})$ defines a new bundle of self-dual $2$-forms, so the splitting \eqref{splitting-2-forms} changes. 
If $\widetilde{F}$ is sufficiently close to $F$, then as in the beginning of Section \ref{section-uniform-estimates} we can write $\SPAN\{\omega,\tilde{\alpha},\tilde{\beta}\}=\Lambda^+_{H,\mu}$ for a map $\mu:\Lambda^+_H\rightarrow \Lambda_H^-$. 
To construct a right inverse $VP(\widetilde{F})$ of $DP(\widetilde{F})$, let $\dot{\psi}\in\mathcal{G}(M)=d\Omega^2(M)$. 
By using the Hodge decomposition with respect to some background Riemannian metric, we find a preferred primitive $d\sigma=\dot{\psi}$. More precisely, $\sigma=d^*\xi$ where $\xi$ is the unique solution of $\Delta \xi = \dot{\psi}$. In particular we have $\|\sigma\|_{s+1} \lesssim \|\dot{\psi}\|_s$ for all $s$. 
To construct a pre-image of $d\sigma$ with respect to the map $DP(\widetilde{F})$, we proceed as in the proof of Proposition \ref{Prop-CokernelDP}. 
Decompose $\sigma$ as
\begin{align*}
\sigma=\sigma_H +\chi\wedge\theta=\sigma_+ + \sigma_-+\chi\wedge\theta, \end{align*}
where $\sigma_H\in\Omega_H^2$, $\sigma_{\pm} = \pi_{\pm}(\sigma_H)\in\Omega^{\pm}_H$ and $\chi\in\Omega_H^1$.
In place of \eqref{crucial-equ}, we now need to solve the perturbed equation
\begin{align*}
D_{\mu}\eta = \sigma_-,
\end{align*}
where $D_{\mu}$ is the operator from Section \ref{section-uniform-estimates}. Under the hypothesis $\mathcal{H}=0$ we have the solution $R_{\mu}\sigma_-$
if $\widetilde{F}$ is in a sufficiently small neighbourhood of $F$. 

Denote by $L_{1}(\widetilde{F})$ the projection of $\Lambda_H^2$ to $\mathbb{R}\omega$, by $L_{2}(\widetilde{F})$ the projection to $\mathbb{R}\tilde{\alpha}$ and by $L_3(\widetilde{F})$ the projection to $\mathbb{R}\tilde{\beta}$.
$\tilde{\alpha}$ defines an isomorphism $H \rightarrow \Lambda_H^1$ via $w\mapsto w\lrcorner \tilde{\alpha}$. Denote the inverse by $L_4(\widetilde{F})$.

Tracing through the proof of Proposition \ref{Prop-CokernelDP}, we find $\gamma\in\Omega^1(M)$, such that 
\begin{align*}
\sigma - d\gamma = K(\widetilde{F})(h_{\alpha}, h_{\beta}, w_H),
\end{align*}
where
\begin{align*}
\sigma 
&= 
d^*\Delta^{-1}\dot{\psi},
\\
h_{\omega}
&=L_1(F)(\sigma_H-d_H R_{\mu} \sigma_-),
\\
h_{\alpha}
&=
L_2(\widetilde{F})(\sigma_H-d_H R_{\mu}\sigma_-),
\\
h_{\beta}
&=
L_3(\widetilde{F})(\sigma_H-d_H R_{\mu}\sigma_-),
\\
w_H
&=
L_4(\widetilde{F})\{\chi+\mathcal{L}_v R_{\mu}\sigma_- - d_H h_{\omega}\}.
\end{align*}
We define the right-inverse $VP(\widetilde{F})$ of $DP(\widetilde{F})$ to be
\begin{align*}
VP(\widetilde{F})\dot{\psi}:=(h_{\alpha},h_{\beta},w_H).
\end{align*}
Next we need to show that the map $VP(\widetilde{F})$ satisfies a tame estimate. 
Because the $\SU(2)$-structure depends on three derivatives of $\widetilde{F}$, the operators $L_i(\widetilde{F})\sigma$, $i=1, \dots, 4$, also depend on 
three derivatives of $\widetilde{F}$,
and are linear and of order $0$ in $\sigma$. By \eqref{2nd-Moser_estimate} for each $k\in\mathbb{N}$ we have under a $\mathcal{C}^3$-bound on $\widetilde{F}$ 
\begin{align}
\|L_i(\widetilde{F})\sigma\|_k 
\lesssim
\|\sigma\|_k + [[u(\widetilde{F})]]_{k+3} \|\sigma\|.
\label{Lambda+Moser}
\end{align}
Together with \eqref{uniform-right-inv-estimate} this gives
\begin{gather*}
\| h_{\alpha}\|_{k}
=
\|L_2(\widetilde{F})(\sigma_H-d_H R_{\mu}\sigma_-)\|_{k}
\\
\lesssim
\|\sigma\|_k + \|R_{\mu}\sigma_-\|_{k+1}
+
[[u(\widetilde{F})]]_{k+3} (\|\sigma\| + \|R_{\mu}\sigma_-\|_1)
\\
\lesssim
\|\sigma\|_{k+2}
+
([[\mu]]_{k+4}+1)\|\sigma\|_l
+
[[u(\widetilde{F})]]_{k+3} \|\sigma\|_l
\\
\lesssim
\|\dot{\psi}\|_{k+1}+([[u(\widetilde{F})]]_{k+7}+1) \|\dot{\psi}\|_{l-1}.
\end{gather*}
We get analogous estimates for $h_{\beta}$ and $h_{\omega}$. 
Thus we get
\begin{gather*}
\|\chi''\|_k
\lesssim
\|\chi\|_k + \|R_{\mu}\sigma_-\|_{k+1} + \|h_{\omega}\|_{k+1}
\\
\lesssim
\|\dot{\psi}\|_{k+2}+([[u(\widetilde{F})]]_{k+8}+1) \|\dot{\psi}\|_{l-1}.
\end{gather*}
Similarly to \eqref{Lambda+Moser} under a $\mathcal{C}^3$-bound on $F$ we have
\begin{align*}
\|L_4(\widetilde{F})\chi\|_k 
\lesssim
\|\chi\|_k + [[u(\widetilde{F})]]_{k+3} \|\chi\|.
\end{align*}
Using this gives
\begin{gather*}
\|w_H\|_{k}
\lesssim
\|\chi''\|_k + [[u(\widetilde{F})]]_{k+3} \|\chi''\|
\\
\lesssim
\|\dot{\psi}\|_{k+2}++([[u(\widetilde{F})]]_{k+8}+1) \|\dot{\psi}\|_{l-1}
+ 
[[u(\widetilde{F})]]_{k+3}
\|\dot{\psi}\|_{l-1}
\\
\lesssim
\|\dot{\psi}\|_{k+2}+([[u(\widetilde{F})]]_{k+8}+1) \|\dot{\psi}\|_{l-1},
\end{gather*}
if $\widetilde{F}$ is sufficiently close to $F$.
Thus we get
\begin{align}
\|VP(\widetilde{F})\dot{\psi}\|_{k} 
\lesssim 
\|\dot{\psi}\|_{k+2}+([[u(\widetilde{F})]]_{k+8}+1) \|\dot{\psi}\|_{l-1}.
\label{tame_estimate}
\end{align} 
Case (2):  In this case the splitting \eqref{splitting-forms} changes and $\SPAN\{\tilde{\omega},\tilde{\alpha},\tilde{\beta}\}$ is not a subspace of $\Lambda_H^2$ any more. However, instead of normalsing as in section \ref{section-structure-in-dim5}, we still can write
\begin{align*}
\widetilde{F}^*\Psi = \theta\wedge\alpha',
\quad
\widetilde{F}^*\hat{\Psi} = \theta\wedge\beta',
\end{align*}
where $\alpha'.\alpha'=\beta'.\beta'>0$ and $\omega,\alpha',\beta'$ are mutually orthogonal.
Thus $(\omega,\alpha',\beta')$ spans a positive definite rank $3$ subbundle of $\Lambda_H^2$.
The estimate follows in case (2) if we replace $\SPAN\{\omega,\tilde{\alpha},\tilde{\beta}\}$
by $\SPAN\{\omega,\alpha',\beta'\}$ in the derivation for the estimate in case (1).

Case (3): Now suppose $\widetilde{H}\neq H$. By a result of Gray
\cite{Gray} there exists a diffeomorphism of $M$ such that $\widetilde{H}=\Phi_* H$.
A manifold chart around the identity in the group $\mathcal{D}(M)$ of diffeomorphisms  is given by the inverse $\nu:\mathcal{V}\subset\mathcal{D}(M)\rightarrow\mathcal{C}^{\infty}(M,TM)$ of the exponential map. 
The map $\widetilde{H}\mapsto \Phi$ is tame \cite[Theorem III.2.4.6]{HamiltonIFT}, which implies that there exists $r$ such that 
\begin{align}
[[\nu(\Phi)]]_n \lesssim [[u(\widetilde{F})]]_{n+r}+1
\label{tame-Gray}
\end{align}
for all $n$.
Furthermore, the composition 
\begin{gather*}
\mathcal{E}(M,Z)\times\mathcal{D}(M) \rightarrow \mathcal{E}(M,Z),
\\
(\widetilde{F},\Phi)\mapsto \widetilde{F}\circ \Phi
\end{gather*}
is a smooth tame map \cite[II.4.4.5]{HamiltonIFT}, so there exists an 
$s$ such that 
\begin{align}
[[u(\widetilde{F}\circ \Phi)]]_n \lesssim [[\nu(\Phi)]]_{n+s} + [[u(\widetilde{F})]]_{n+s}+1 \lesssim [[u(\widetilde{F})]]_{n+r+s}+1.
\label{tame-composition}
\end{align}
By the discussion below definition \ref{Def-pseudoconvex}, $(\widetilde{F}\circ \Phi)^*\Psi = \Phi^*\psi$ is a strongly pseudoconvex $3$-form which determines the contact distribution $\Phi^{-1}_*\widetilde{H}= H$. Furthermore, if $\widetilde{F}$ is sufficiently close to $F$, then $\Phi$ must be sufficiently close to the identity. Thus we can apply the result from cases (1) and (2) to $\widetilde{F}\circ \Phi$, i.e. that there exists a right inverse $VP(\widetilde{F}\circ \Phi)$ to $DP(\widetilde{F}\circ \Phi)$ which satisfies the estimate \eqref{tame_estimate}. The calculation
\begin{align*}
d(X\lrcorner \Phi^*\psi) = d(\Phi^*((\Phi_*X)\lrcorner \psi)) = \Phi^*d((\Phi_*X)\lrcorner\psi))
\end{align*}
shows that $VP(\widetilde{F}):=\Phi_*\circ VP(\widetilde{F}\circ \Phi)\circ \Phi^*$ is a right inverse for $DP(\widetilde{F})$.
A repeated application of the second Moser estimate and  \eqref{tame_estimate} 
gives
\begin{align*}
&\|VP(\widetilde{F})\dot{\psi}\|_k
=
\|\Phi_* VP(\widetilde{F}\circ \Phi)\Phi^*\dot{\psi}\|_k
\\
\lesssim
&
\|VP(\widetilde{F}\circ \Phi)\Phi^*\dot{\psi}\|_k
+
[[\nu(\Phi)]]_{k+1} \|VP(\widetilde{F}\circ \Phi)\Phi^*\dot{\psi}\|
\\
\lesssim
&
\|\Phi^*\dot{\psi}\|_{k+2}+ ([[u(\widetilde{F}\circ \Phi)]]_{k+8}+1)
\|\Phi^*\dot{\psi}\|_{l-1}
\\
&+
[[\nu(\Phi)]]_{k+1}
(\|\Phi^*\dot{\psi}\|_{2}+ ([[u(\widetilde{F}\circ \Phi)]]_{8}+1)
\|\Phi^*\dot{\psi}\|_{l-1})
\\
\lesssim
&
\|\dot{\psi}\|_{k+2}
+[[\nu(\Phi)]]_{k+3} \|\dot{\psi}\|
+
([[u(\widetilde{F}\circ \Phi)]]_{k+8}+1)
(\|\dot{\psi}\|_{l-1} + [[\nu(\Phi)]]_{l} \|\dot{\psi}\|)
\\
&+
[[\nu(\Phi)]]_{k+1}
(\|\dot{\psi}\|_{2}+[[\nu(\Phi)]]_3 \|\dot{\psi}\| + ([[u(\widetilde{F}\circ \Phi)]]_{8}+1)
(\|\dot{\psi}\|_{l-1}+[[\nu(\Phi)]]_l \|\dot{\psi}\|)).
\end{align*} 
With the observations \eqref{tame-Gray} and \eqref{tame-composition}
for $\widetilde{F}$ in a sufficiently small neighbourhood of $F$ 
we get
\begin{align}
\|VP(\widetilde{F})\dot{\psi}\|_k 
\lesssim
\|\dot{\psi}\|_{k+2} + ([[u(\widetilde{F})]]_{k+d}+1) \|\dot{\psi}\|_{l-1}
\label{tame-estimate-final}
\end{align}
for some $d$.
Putting together the results for the three cases, we have constructed a right inverse $VP(\widetilde{F})$ of $DP(\widetilde{F})$ for all $\widetilde{F}$ in some neighbourhood of $F$ in $\mathcal{E}(M,Z)$ which for all $k$ satisfies an estimate of the form \eqref{tame-estimate-final}. 
The result follows from \cite[III.1.1.3]{HamiltonIFT}
\end{proof}

\subsubsection*{Local uniqueness}

To complete the proof of Theorem \ref{mainTHM}, we are left to prove the local uniqueness statements for embeddings into $\mathbb{C}^3$. We start by proving local uniqueness for Problem \ref{CalabiProblem}.

\begin{proposition}
\label{uniqueness-Calabi}
Let $(\theta,\omega,\alpha,\beta)$ be a contact hyperk\"ahler $\SU(2)$-structure  on $M$ with $\mathcal{H}=0$. Then there exists a neighbourhood of $\beta$ in $\Omega^2_H$ in which $\beta$ is the only solution of the system
\begin{align*}
\omega.\tilde{\beta}=0,
\quad
\alpha.\tilde{\beta}=0,
\quad
\tilde{\beta}.\tilde{\beta}=1,
\quad
d_H\tilde{\beta}=0.
\end{align*}
\end{proposition}
\begin{proof}
Suppose there is another solution $\tilde{\beta}$. Then the first two equations imply that we can decompose $\tilde{\beta}=h\beta + \tilde{\beta}_-$, where $\tilde{\beta}_-\in\Omega^-_H$.
The third condition gives $\tilde{\beta}_-.\tilde{\beta}_-=1-h^2$.
If we choose the neighbourhood of $\beta$ sufficiently small such that $h > -1$, then $\SPAN\{\omega,\alpha,(1+h)\beta+\tilde{\beta}_-\}$ is a positive definite subspace which defines a perturbation $\Lambda^+_{H,\mu}$ of $\Lambda^+_H$. It is constructed such that the difference $\gamma:= \tilde{\beta}-\beta$ is a section of $\Lambda^-_{H,\mu}$. The last equation implies  $d_H\gamma=0$ and thus $\gamma\in \ker\Box_{H,\mu}$. Because the vanishing of $\ker\Box$ is an open condition, 
$\gamma$ must vanish if $\mu$ is sufficiently small.
\end{proof}

Suppose that $\widehat{F}$ is another embedding of $M$ into $\mathbb{C}^3$
which is close to $F$ and satisfies $\widehat{F}^*\Psi=F^*\Psi=\psi$. Then by Proposition \ref{uniqueness-Calabi} we also have $\widehat{F}^*\hat{\Psi}=F^*\hat{\Psi}$. Denote by $\phi$ the diffeomorphism $\widehat{F}\circ F^{-1}$ from $F(M)$ to $\widehat{F}(M)$. Because $\phi$ preserves $\alpha+i\beta$ it is an isomorphism of CR-manifolds. In particular, the component functions of $\phi=(\phi_1,\phi_2,\phi_3)$ are CR functions, i.e. are annihilated by $\bar{\partial}_b$. Because $F(M)$ is strongly pseudoconvex, $\phi_1,\phi_2$ and $\phi_3$ extend to holomorphic functions
$\Phi_1,\Phi_2,\Phi_3$ on the domain $U$ which is bounded by $F(M)$ \cite[Theorem 5.3.2]{FollandKohn}. We are left to prove that $\Phi=(\Phi_1,\Phi_2,\Phi_3)$ preserves $\Psi$. We can write $\Phi^*(\Psi+i\hat{\Psi})=h (\Psi+i\hat{\Psi})$ for some holomorphic function on $U$. In particular, the real and imaginary parts of $h$ are harmonic with respect to the Euclidean metric on $\mathbb{C}^3$. On the boundary we have $\Re h =1$ and $\Im h =1$. By the maximum principle we have $\Re h =1$ and $\Im h =1$ on all of $U$. Thus $\Phi^*(\Psi+i\hat{\Psi})=\Psi+i\hat{\Psi}$.

\appendix

\section{Vanishing of cohomology for pseudoconvex submanifolds of Stein Manifolds}

\label{appendix-vanishing}

Let $Z$ be a complex manifold with complex dimension $n$ and boundary $M:=\partial Z$. Denote by $\mathcal{A}^{p,q}(Z)$ the space of $(p,q)$-forms on $Z$ which are smooth up to the boundary.
Choose a hermitian metric on $Z$, which then also induces a hermitian metric on the bundles $\Lambda^{p,q}Z$.
Denote by $\Lambda^{p,q}Z|_M$ the restriction of the bundle $\Lambda^{p,q}Z$ to $M$, i.e. the collection of the spaces $\Lambda^{p,q}Z|_z$, where $z$ varies over $M$. 
A $(p,q)$-form at $z\in M$ on the boundary is called \textit{complex normal}
if it is of the form $\bar{\partial}r\wedge \zeta$, where $\zeta\in\Lambda^{p,q-1}Z|_M$ and $r$ is a boundary defining function in a neighbourhood of $z$.
Denote by $\Lambda_N^{p,q}M|_z \subset \Lambda^{p,q}Z|_z$ the space of 
all complex normal $(p,q)$-forms at $z$. 
$\Lambda_N^{p,q}M|_z$ does not depend on the choice of $r$.
Denote by $\Lambda^{p,q}_T M|_z$ the orthogonal complement of $\Lambda_N^{p,q}M|_z$ in $\Lambda^{p,q}Z|_z$.
Forms in $\Lambda_T^{p,q}M|_z$ are called \textit{complex tangential}.
We note that $\Lambda_N^{p,q}M|_z$ is not a subspace of $\Lambda^{p+q} T^*_z M \otimes \mathbb{C}$.
By $\Lambda^{p,q}_N M$ and $\Lambda^{p,q}_T M$ we denote the bundles over $M$ given by the union of the $\Lambda^{p,q}_N M|_z$ and $\Lambda^{p,q}_T M|_z$, respectively. We will denote the spaces of smooth sections of 
$\Lambda^{p,q}_N M$ and $\Lambda^{p,q}_T M$ by 
$\mathcal{A}^{p,q}_N(M)$ and $\mathcal{A}^{p,q}_T(M)$, respectively.
$\mathcal{A}^{p,q}_N(Z)$ are those $(p,q)$-forms on $Z$ which restricted to the boundary lie in $\mathcal{A}^{p,q}_N(M)$, and 
$\mathcal{A}^{p,q}_T(Z)$ are those $(p,q)$-forms on $Z$ which restricted to the boundary lie in $\mathcal{A}^{p,q}_T(M)$.

Because of $\bar{\partial}^2=0$ we have for each $p=0,\dots, n$ a chain complex $(\mathcal{A}^{p,q}(Z),\bar{\partial})_q$, whose cohomology groups we denote by $H^{p,q}(Z)$. A form in $\mathcal{A}^{p,q}_N(Z)$ can be written as $\bar{\partial}r\wedge \sigma + r\xi$, where $\sigma\in\mathcal{A}^{p,q-1}(Z)$ and $\xi\in\mathcal{A}^{p,q}(Z)$. The calculation
\begin{align}
\bar{\partial}(\bar{\partial}r\wedge \sigma + r\xi)
=
\bar{\partial}r\wedge(\xi-\bar{\partial}\sigma) + r \bar{\partial}\xi
\label{del_bar_preserves_N}
\end{align}
shows that $\bar{\partial}$ maps $\mathcal{A}^{p,q}_N(Z)$ to 
$\mathcal{A}^{p,q+1}_N(Z)$. Denote the cohomology groups of the complex
$(\mathcal{A}^{p,q}_N(Z),\bar{\partial})_q$ by $H^{p,q}_N(Z)$.

Next we define the $\bar{\partial}_b$-complex on the boundary $M$.
If $\zeta\in\mathcal{A}^{p,q}_T(M)$, let $\hat{\zeta}$ be any extension to $\mathcal{A}^{p,q}(Z)$. Set $\bar{\partial}_b \zeta:=\pi^T_M(\bar{\partial}\hat{\zeta})$, where $\pi^T_M$ denotes the projection from $\mathcal{A}^{p,q}(Z)$ to $\mathcal{A}^{p,q}_T(M)$.
Any two extensions of $\zeta$ differ by an element in $\mathcal{A}^{p,q}_N(Z)$. 
By \eqref{del_bar_preserves_N} $\pi^T_M\circ \bar{\partial}$ vanishes on 
$\mathcal{A}^{p,q}_N(Z)$, so that $\bar{\partial}_b$ is well-defined. $\bar{\partial}^2=0$ implies $\bar{\partial}_b^2=0$, and thus we get a chain complex $(\mathcal{A}^{p,q}_T(M),\bar{\partial}_b)_q$, whose cohomology groups we denote by $H^{p,q}(M)$. By the construction of $\bar{\partial}_b$ we have a short exact sequence of chain complexes
\begin{center}
\begin{tikzcd}[row sep=scriptsize, column sep=scriptsize]
0 \arrow[r] 
&
\mathcal{A}^{p,q+1}_N(Z)
\arrow[r,hookrightarrow] 
&
\mathcal{A}^{p,q+1}(Z)
\arrow[r,"\pi^T_M"] 
&
\mathcal{A}^{p,q+1}_T(M)
\arrow[r] 
&
0
\\
0 \arrow[r] 
&
\mathcal{A}^{p,q}_N(Z)
\arrow[r,hookrightarrow]
\arrow[u,"\bar{\partial}"] 
&
\mathcal{A}^{p,q}(Z)
\arrow[r,"\pi^T_M"]
\arrow[u,"\bar{\partial}"] 
&
\mathcal{A}^{p,q}_T(M)
\arrow[r] 
\arrow[u,"\bar{\partial}_b"]
&
0
\end{tikzcd}
\end{center}
This induces a long exact sequence of cohomology groups
\begin{center}
\begin{tikzcd}
\cdots 
\arrow[r]
&
H^{p,q}_N(Z)
\arrow[r]
&
H^{p,q}(Z)
\arrow[r]
&
H^{p,q}(M)
\arrow[r]
&
H^{p,q+1}_N(Z)
\arrow[r]
&
\cdots
\end{tikzcd}
\end{center}

\begin{proposition}\cite[Proposition 5.1.5]{FollandKohn}
If the Levi form of $M$ has at least $n-q$ positive eigenvalues or at least $q+1$ negative eigenvalues (certainly true on a strongly pseudoconvex domain), then $H^{p,q}(Z)\cong H^{n-p,n-q}_N(Z)^*$.
\end{proposition}

\begin{corollary}
\label{corollary-Stein-vanishing}
Let $Z$ be a strongly pseudoconvex domain in a Stein manifold of complex dimension $3$ with boundary $M$. Then $H^{1,1}(M)=0$.
\end{corollary}
\begin{proof}
Because $Z$ is a strongly pseudoconvex domain in a Stein manifold we have $H^{1,1}(Z)=0$ and $H^{1,2}_N(Z)\cong H^{2,1}(Z)^*=0$ \cite[Corollary 5.2.6]{Hormander_book}. The long exact sequence gives $H^{1,1}(M)=0$.
\end{proof}

\section{Alternative proof that the vanishing of $\mathcal{H}$ is an open condition}

\label{appendix-B}

In this section we prove in a more quantitative way that the obstruction space vanishes for an open set of Euclidean metrics on $H$. The advantage of this approach is that it can give an estimate on the size of the neighbourhood of $g_H$ in which the obstruction space vanishes. We will derive an explicit such bound in the case of the standard embedding of $S^5$ in $\mathbb{C}^3$.

Suppose that $\mathcal{H}=0$, so that $\Box_H$ is invertible and $d_H^-$ has right inverse $R=d_H^*\Box_H^{-1}$ as in Proposition \ref{Prop-Fredholm-alternative}.
As in Section \ref{section-uniform-estimates} we parametrise nearby metrics on $H$ by a bundle map $\mu:\Lambda_H^+\rightarrow\Lambda_H^-$.
The vanishing of $\ker \Box_{H,\mu}$ is equivalent to the surjectivity of $d_{H,\mu}^-$.
We have the explicit formula 
\begin{align*}
(\mathbbm{1}-\mu\mu^*)\circ \pi^-\circ d^-_{H,\mu}
=
d^-_H-\mu \circ d^+_H.
\end{align*}
$\mathbbm{1}-\mu\mu^*$ is an isomorphism if $|\mu|<1$. Thus under this bound on $\mu$, it is enough to show that 
\begin{align*}
d^-_{H}-\mu\circ d^+_H: \Omega_H^1\rightarrow \Omega_H^-
\end{align*}
is surjective. To prove this, we will show that 
\begin{align*}
(d^-_{H}-\mu\circ d^+_H)\circ R = \mathbbm{1}-\mu\circ d_H^+ \circ R: L^2(\Lambda_H^-)\rightarrow L^2(\Lambda_H^-)
\end{align*}
is an isomorphism if $\mu$ is sufficiently small. 
If $\kappa$ is a constant such that, for all $\sigma$, 
\begin{align}
\Vert d^{+}_{H} R \sigma \Vert \leq \kappa \Vert \sigma\Vert
\label{kappa} 
\end{align}
then if $\vert \mu\vert<\kappa^{-1}$ everywhere the operator 
$$    \mu d^{+}_{H} R: L^{2}(\Lambda^{-}_{H})\rightarrow L^{2}(\Lambda^{-}_{H})$$ has operator norm less than $1$ and it follows from the usual Neumann series that $d^{-}_{H,\mu}$ is surjective.

On a compact, oriented Riemannian $4$-manifold we have the identity $\|d^+\eta\|=\|d^-\eta\|$ for any $1$-form $\eta$. In the contact case there is an extra term:
\begin{align*}
\|d_H^+\eta\|^2-\|d_H^-\eta\|^2
&=
\int_M d_H \eta \wedge d_H\eta\wedge\theta
=
\int_M d\eta\wedge d\eta\wedge \theta
=
\int_M d\eta\wedge \eta\wedge d\theta
\\
&=
\int_M d\eta\wedge \eta\wedge\omega
=
\int_M \theta\wedge \mathcal{L}_v\eta\wedge\eta\wedge\omega
=
-\int_M \langle I \mathcal{L}_v\eta,\eta\rangle \Vol 
\\
&= 
-( I \mathcal{L}_v\eta,\eta).
\end{align*}
Setting $\eta= R\sigma$, this gives
\begin{align}
\Vert d^{+}_{H} R \sigma\Vert^{2}\leq   \Vert \sigma\Vert^{2} + \vert ( I \cL_{v} \eta,\eta ) \vert.
\label{d+estimate-step1}
\end{align}

The existence of a constant $\kappa$ in \eqref{kappa} follows from formula \eqref{d+estimate-step1} and our analysis in Section \ref{section_linear_analysis}.
Let $P$ be the standard self-adjoint second order operator $P= 1 + \nabla^{*}\nabla$.
then
$$ \vert ( I \cL_{v} \eta,\eta )\vert = \vert ( P^{-1/4} I\cL_{v} \eta, P^{1/4} \eta ) \vert \leq c_{1} \Vert \eta \Vert^{2}_{\frac{1}{2}}, $$
since $P^{-1/4} I \cL_{v}$ and $P^{1/4}$ are pseudo-differential operators of order $1/2$. Now the estimate \eqref{improved-estimate} for the operator $R$ gives
$$ \vert ( I \cL_{v} \eta,\eta )\vert \leq c_{2} \Vert \sigma \Vert^{2}, $$
for some $c_{2}$. Then we can take $\kappa= \sqrt{1+c_{2}}$.

We now study the case of the standard contact structure on $S^{5}$, with standard metric. Regarding $S^{5}$ as the principle circle bundle over $\bC\bP^{2}$ corresponding to the Hopf line bundle $L\rightarrow \bC \bP^{2}$, the contact structure is the field of horizontal subspaces for the standard connection. Passing to the complexified bundles, we can write any $\eta\in\Omega^{1}_{H}$ as a sum $\eta=\sum_{k} \eta_{k}$ of components in the weight spaces for the circle action.  A component $\eta_{k}$ can be regarded as a $1$-form on $\bC\bP^{2}$ with values in the line bundle $L^{k}$. 
Similarly for sections of $\Lambda^{-}_{H}$. The operator $d^{-}_{H}$ is a sum of components
$$    d^{-}_{k}:\Omega^{1}_{\bC\bP^{2}}(L^{k})\rightarrow \Omega^{-}_{\bC\bP^{2}}(L^{k}), $$ and similarly $\Box= \sum_k \Box_{k}$. These operators are very familiar in 4-dimensional Riemannian geometry and in particular the fact that $\bC\bP^{2}$ is a self-dual manifold and that the curvature of $L$ is a self-dual form means that there are   simple Weitzenbock formulae
$$   \Box_{k}= \frac{1}{2} \nabla_{k}^{*}\nabla_{k} + S/6, $$
where $S$ is the scalar curvature of $\bC\bP^{2}$. 
Because $S$ is positive, this gives an alternative proof that the obstruction space $\ker\Box$ vanishes for this example. 

To get   favourable bounds on the operators $\nabla_{k}^{*}\nabla_{k}$ we review some general theory. Let $V$ be a complex vector bundle with Hermitian connection over a compact K\"ahler manifold $X$. The connection defines operators
$$  \partial_{V}:\Omega^{0}(V)\rightarrow \Omega^{1,0}(V) \ \ \ \ \, \ \ \ \   \db_{V}:\Omega^{0}(V)\rightarrow \Omega^{0,1}(V), $$
in the usual way (note that we not supposing that $\db_{V}$ defines a holomorphic structure). Then we have the identity, for any section $s$ of $V$: 
$$  \Vert \partial_{V} s \Vert^{2}- \Vert \db_{V} s \Vert^{2} = ( i (F_{V}.\omega)(s), s) ). $$
where $F_{V}$ is the curvature of the connection. In our application we take
$V=\Lambda^{-}_{\bC\bP^{2}}\otimes L^{k}$. The fact that $\bC\bP^{2}$ is an Einstein manifold implies that the curvature of the bundle $\Lambda^{-}_{\bC\bP^{2}}$ is anti-self-dual and so orthogonal to the self-dual form $\omega$. Thus the only contribution to $F_{V}.\omega$ comes from $L^{k}$ and we get
$$   \Vert \partial_{V} s \Vert^{2}- \Vert \db_{V} s \Vert^{2} = k \Vert s \Vert^{2}. $$
This implies that
$$ \Vert \nabla_{k} s \Vert^{2} \geq \vert k\vert \Vert s\Vert^{2}. $$
Combining with the previous discussion, and using the fact that the scalar curvature is positive we get
$$ ( \Box_{k}s,s) \geq \frac{1}{2} \vert k\vert \Vert s\Vert^{2}. $$
This implies that, for $k\neq 0$,
$$  ( \Box_{k}^{-1}s,s) \leq 2 \vert k\vert^{-1} \Vert s\Vert^{2}. $$
Now if $\sigma\in \Omega^{-}_{H}$ has components $ \sigma_{k}$ and $\eta_{k}= d^{*}_k\Box_{k}^{-1}\sigma_{k}$ we have
$$  \Vert \eta_{k}\Vert^{2}= ( d^{*}\Box_{k}^{-1}\sigma_{k},d^{*}\Box_{k}^{-1}\sigma_{k})=  ( \Box_{k}^{-1}\sigma_{k},\sigma_{k}) \leq 2 \vert k\vert^{-1} \Vert \sigma_{k} \Vert^{2}.  $$
One can check that the Reeb vector field $v$ is $\sqrt{2}$ times the generator of the circle action. Therefore, the Lie derivative term is $\cL_{v} \eta_{k}= i \sqrt{2} k \eta_{k}$ so
$$ \vert \langle I \cL_{v} \eta, \eta\rangle \vert = \vert \sum \langle \sqrt{2} k I i \eta_{k}, \eta_{k} \rangle \vert \leq \sum \sqrt{2}\vert k \vert \Vert \eta_{k}\Vert^{2}. $$
Thus we get
$$   \vert \langle I \cL_{v} \eta, \eta\rangle \vert\leq \sum_k 2\sqrt{2}\Vert \sigma_{k}\Vert^{2}= 2\sqrt{2}\Vert \sigma \Vert^{2}.$$
We can take $c_{2}=1$ and we see that $d_{\mu}$ is invertible if
$\vert \mu\vert < 1/\sqrt{1+2\sqrt{2}}$.

\section{A Weitzenb\"ock formula for an SU(2)-holonomy connection}
\label{appendix-C}

In Lemma \ref{lemma-connection} we have shown that there exist metric connections $\nabla$ on $TM$ which respect the contact structure in a suitable way. From Lemma \ref{lemma-general_Weitzenbock}
and the proof of Proposition \ref{Proposition-1/2-estimate-Q}
 we know that for each such a connection there is a self-adjoint endomorphism $\mathcal{R}^{\nabla}$ of $\Lambda^-_H$ such that 
\begin{align*}
\Box_H=\frac{1}{2} \nabla^*_H\nabla_H + \mathcal{R}^{\nabla}.
\end{align*}
This gives us criterion for the vanishing of the obstruction space $\mathcal{H}$.
\begin{proposition}
If $\mathcal{R}^{\nabla}$ is a positive operator, then $\mathcal{H}=0$. 
\end{proposition}
In $4$-dimensional Riemannian geometry, the curvature term in the Weitzenb\"ock formula is 
\begin{align}
W^-+\frac{s}{3},
\label{Weitzenbock-Riemann}
\end{align}
where $W^-$ is the anti-self-dual part of the Weyl tensor and $s$ is the scalar curvature. 
The goal of this Appendix is to pick a particularly well-suited connection from  Lemma \ref{lemma-connection} and show that for this connection $\mathcal{R}^{\nabla}$ essentially is a contact analogue of \eqref{Weitzenbock-Riemann}. The choice in Lemma \ref{lemma-connection} is $T_v^a$, the skew-symmetric part of $T(v,\cdot)$, the torsion tensor contracted with $v$. In Section \ref{torsion} we show that the additional condition that the full $\SU(2)$-structure is parallel completely determines the component of $T_v^a$ in $\mathfrak{so}^+(H)$. We then set the component in $\mathfrak{so}^-(H)$ equal to $0$ to pick a specific such connection. In Section \ref{section-curvature} we describe the curvature tensor of this connection, and in Section \ref{section-Weitzenbock} we compute the Weitzenb\"ock formula.

\subsection{Determining the torsion}

\label{torsion}

In this section we prove

\begin{proposition}
\label{proposition-connection}
$(\omega,\alpha,\beta)$ are parallel with respect to a connection $\nabla$ on $TM$ from Lemma \ref{lemma-connection}
if and only if 
\begin{align*}
T^a_v=-\frac{1}{2}f I \mod \Gamma(\mathfrak{so}^-(H)).
\end{align*}
\end{proposition}
By property (ii) of Lemma \ref{lemma-connection} we have
\begin{align*}
\Alt(\nabla_H\omega)=d_H \omega = 0,
\quad
\Alt(\nabla_H\alpha)=d_H \alpha = 0,
\quad
\Alt(\nabla_H\beta)=d_H\beta =0.
\end{align*}
With a similar representation theoretic argument as for hyperk\"ahler structures \cite{Salamon-book}, one can conclude
\begin{align*}
\nabla_H\omega = 0, 
\quad
\nabla_H\alpha = 0,
\quad
\nabla_H\beta =0.
\end{align*}
Thus we need to determine $T_v^a$ such that
\begin{align*}
\nabla_v\omega=0,
\quad
\nabla_v\alpha=0,
\quad
\nabla_v\beta=0.
\end{align*}

It will be helpful to express the Lie derivatives of $g_H, \omega,\alpha,\beta$ along $v$ as endomorphism fields with respect to the metric $g_H$.

\begin{lemma}
As endomorphism fields, the Lie derivatives of $g, \alpha$ and $\beta$ in the direction of $v$ correspond to
\begin{gather}
\Lv g \sim \Lv I \circ I = -I \circ \Lv I,
\label{Lie-v-g}
\\
\Lv \alpha \sim  \Lv I \circ K + \Lv J=-I \circ \Lv K,
\label{Lie-v-alpha}
\\
\Lv \beta \sim -\Lv I \circ J + \Lv K = I \circ \Lv J.
\label{Lie-v-beta}
\end{gather}
\end{lemma}
\begin{proof}
Because $\Lv\omega=0$, for any $X,Y\in\Gamma(H)$ we get
\begin{align*}
0=\Lv\omega(X,Y) &=v(g(IX,Y))-g(I\Lv X, Y)-g(IX,\Lv Y)
\\
&=
v(g(IX,Y))-g(\Lv (IX), Y)-g(IX,\Lv Y)+g(\Lv(IX)-I\Lv X,Y)
\\
&=
\Lv g(IX,Y)+g((\Lv I)X,Y).
\end{align*}
Thus $\Lv g$ corresponds to the endomorphism $\Lv I\circ I$. 
An analogous computation for $\alpha$ and $\Lv g_H \sim \Lv I\circ I$ give
\begin{align*}
\Lv \alpha(X,Y) &= \Lv g(JX,Y)+g((\Lv J)X,Y)
\\
&=
g((\Lv I \circ I \circ J)X,Y)+g((\Lv J)X,Y)
\\
&=
g((\Lv I \circ K + \Lv J)X,Y).
\end{align*}
Similarly  for $\beta$ we get
\begin{align*}
\Lv \beta(X,Y) &= \Lv g(KX,Y)+g((\Lv K)X,Y)
\\
&=
g((\Lv I \circ I \circ K)X,Y)+g((\Lv K)X,Y)
\\
&=
g((-\Lv I \circ J + \Lv K)X,Y).
\end{align*}
The remaining identites follow from
\begin{gather*}
0=\Lv(-\mathbbm{1})=\Lv I^2=\Lv I \circ I + I \circ \Lv I,
\\
\Lv J = - \Lv(I\circ K)= - \Lv I \circ K -I \circ \Lv K,
\\
\Lv K = \Lv(I\circ J)= \Lv I\circ J + I\circ \Lv J.
\end{gather*}
\end{proof}

\begin{lemma}
The conditions
\begin{align*}
\nabla_v\omega=0,
\quad
\nabla_v\alpha=0,
\quad
\nabla_v\beta=0,
\end{align*}
are equivalent to 
\begin{subequations}
\begin{align}
[I, T_v^a] &= 0, 
\label{cond1}
\\
[J, T_v^a] &= \frac{1}{2} [J, \Lv K \circ K],
\label{cond2}
\\
[K, T_v^a] &= \frac{1}{2} [K, \Lv J \circ J]. 
\label{cond3}
\end{align}
\end{subequations}
\end{lemma}
\begin{proof}
If $B$ is any field of bilinear forms on $H$ and $X,Y\in\Gamma(H)$, formula \ref{covariant_v_derivative} implies
\begin{align*}
(\nabla_v B)(X,Y) &= v(B(X,Y))-B(\nabla_v X,Y)-B(X,\nabla_v Y)
\\
&=
v(B(X,Y))-B(\Lv X,Y)-B(X,\Lv Y)-B(T_v X,Y)-B(X, T_v Y)
\\
&=
\Lv B(X,Y)- T_vB(X,Y),
\end{align*}
where we write 
\begin{align*}
T_vB(X,Y)= B(T_v X,Y)+B(X,T_v Y)
\end{align*}
for the action of $T_v$ on tensor fields. Thus we need to determine $T_v^a$ such that
\begin{align*}
T_v \omega = \Lv \omega =0, \quad T_v \alpha = \Lv \alpha, \quad T_v \beta = \Lv \beta.
\end{align*}
The first equation gives
\begin{align*}
0 = T_v \omega(X,Y) = g(IT_v X,Y)+g(IX,T_v Y)
=
g(IT_v X,Y)-g(X,IT_v Y) = 2 g(\mathrm{Alt}(I\circ T_v)X,Y).
\end{align*}
We have 
\begin{align*}
2 \mathrm{Alt}(I\circ T_v)
&=
(I\circ T_v)-(I\circ T_v)^*
=
I \circ T_v + T_v^*\circ I
=
I \circ T_v^s + T_v^s \circ I + I \circ T_v^a - T_v^a \circ I 
\\
&=
\frac{1}{2} \Lv I-\frac{1}{2}\Lv I + I \circ T_v^a - T_v^a\circ I
=
[I,T_v^a]
.
\end{align*}
Thus $\nabla_v \omega =0$ gives  condition 
\eqref{cond1}.
$\nabla_v\alpha=0$ leads to the equation
\begin{align*}
\Lv\alpha(X,Y)
&=
T_v \alpha(X,Y) = g(J T_v X,Y)+g(JX,T_v Y) 
\\
&= g(J T_v X,Y)-g(X, JT_v Y)
\\
&=
2g(\mathrm{Alt}(J\circ T_v)X,Y),
\end{align*}
and analogously $\nabla_v\beta=0$ gives
\begin{align*}
\Lv \beta(X,Y) = 2 g(\mathrm{Alt}(K\circ T_v)X,Y).
\end{align*}
Substituting formulas \eqref{Lie-v-alpha} and \eqref{Lie-v-beta}
for $\Lv\alpha$ and $\Lv\beta$ in the above equations gives
\begin{align*}
[J,T_v^a]
&=
2 \mathrm{Alt}(J\circ T_v^a) 
= 
\Lv I \circ K +\Lv J -2\mathrm{Alt}(J\circ T_v^s) 
\\
&=
\Lv I \circ K +\Lv J - (J \circ T_v^s + T_v^s \circ J)
\\
&=
\Lv I \circ K + \Lv J + \frac{1}{2} J \circ I \circ \Lv I - \frac{1}{2} \Lv I \circ I \circ J
\\
&=
\frac{1}{2} \Lv I \circ K -\frac{1}{2}K \circ \Lv I + \Lv J
\\
&=
\frac{1}{2}(\Lv J \circ K + J \circ \Lv K)\circ K 
+
\frac{1}{2}K\circ(\Lv K \circ J + K \circ \Lv J) + \Lv J
\\
&=
\frac{1}{2} J \circ \Lv K \circ K + \frac{1}{2} K \circ \Lv K \circ J
\\
&=
\frac{1}{2}[J, \Lv K \circ K]
\end{align*}
and
\begin{align*}
[K,T_v^a]
&=
2 \mathrm{Alt}(K\circ T_v^a) 
= 
-\Lv I \circ J + \Lv K - 2\mathrm{Alt}(K\circ T_v^s) 
\\
&=
-\Lv I \circ J + \Lv K - (K \circ T_v^s + T_v^s \circ K)
\\
&=
-\Lv I \circ J + \Lv K + \frac{1}{2} K \circ I \circ \Lv I - \frac{1}{2} \Lv I \circ I \circ K
\\
&=
\frac{1}{2} J \circ \Lv I  -\frac{1}{2} \Lv I \circ J + \Lv K
\\
&=
\frac{1}{2}J \circ (\Lv J \circ K + J \circ \Lv K)
+
\frac{1}{2}(\Lv K \circ J + K \circ \Lv J)\circ J + \Lv K
\\
&=
\frac{1}{2} J \circ \Lv J \circ K + \frac{1}{2} K \circ \Lv J \circ J
\\
&=
\frac{1}{2}[K, \Lv J \circ J].
\end{align*}
Therefore, $\nabla_v\alpha=0$ and $\nabla_v\beta=0$ are equivalent to \eqref{cond2} and \eqref{cond3}, respectively.
\end{proof}

We now compute the right hand sides in \eqref{cond2} and \eqref{cond3}. 
The identity $\mathcal{L}_v\alpha = f\beta\mod \Omega_H^-$ and \eqref{Lie-v-alpha} give
\begin{align*}
\Lv \alpha \sim - I \circ \Lv K = f K \mod \Gamma(\mathfrak{so}^-(H)),
\end{align*}
and thus
\begin{align*}
\Lv K \circ K = -f I \mod\, J\circ \Gamma(\mathfrak{so}^-(H)).
\end{align*}
Because $J$ commutes with $\mathfrak{so}^-(H)$, we get 
$[J, \Lv K \circ K ] = 2f K$.
Analogously, $\Lv\beta=-f \alpha\mod \Omega^-_H$ and \eqref{Lie-v-beta} give
\begin{align*}
\Lv \beta \sim I \circ \Lv J = -f J \mod \Gamma(\mathfrak{so}^-(H))
\end{align*}
and thus
\begin{align*}
\Lv J \circ J = -f I \mod\,K\circ \Gamma(\mathfrak{so}^-(H)),
\end{align*}
which implies $[K, \Lv J \circ J] = -2f J$. 
Thus to determine $T_v^a$ we need to solve the system 
\begin{subequations}
\begin{align}
[I, T_v^a] &= 0, 
\\
[J, T_v^a] &= f K,
\\
[K, T_v^a] &= - f J. 
\end{align}
\end{subequations}
The solution is $T_v^a= -\frac{1}{2}f I \mod \Gamma(\mathfrak{so}^-(H))$,
which proves Proposition \ref{proposition-connection}.

\subsection{Structure of the curvature tensor}

\label{section-curvature}

From now on we fix a connection on $TM$ with the properties from Lemma \ref{lemma-connection} and Proposition \ref{proposition-connection} by setting the component of $T_v^a$ in $\mathfrak{so}^-(H)$ equal to zero and get
$T_v=\frac{1}{2}\Lv I \circ I -\frac{1}{2}f I$.
Denote by $R_g$ and $R^{\nabla}$ the curvature tensors 
of the Levi-Civita connection of $g$ and of $\nabla$, respectively.
The following Lemma gives the analogue of the Gauss equations in submanifold theory, relating $R_g$ and $R^{\nabla}$ in terms of the ``second fundamental form'' $B$ introduced in Lemma \ref{lemma-connection}.

\begin{lemma}
\label{Gauss-equation}
For $X,Y,Z,W\in\Gamma(H)$ we have 
\begin{align*}
\pi_H(R_g(X,Y)Z)
&=
R^{\nabla}(X,Y)Z + B(Y,Z) \nabla^{LC}_X v - B(X,Z) \nabla^{LC}_Y v 
\\
&+ \omega(X,Y) \nabla^{LC}_Z v - \omega(X,Y) T_v Z,
\\
R_g(X,Y,Z,W)
&=
R^{\nabla}(X,Y,Z,W) 
-B(Y,Z)B(X,W)+B(X,Z) B(Y,W)
\\
&+\frac{f+1}{2} \omega(X,Y)\omega(Z,W).
\end{align*}
\end{lemma}
\begin{proof}
We compute
\begin{align*}
\nabla^{LC}_X \nabla^{LC}_Y Z 
&=
\nabla^{LC}_X (\nabla_Y Z + B(Y,Z)v)
\\
&=
\nabla_X \nabla_Y Z + B(Y,Z) \nabla^{LC}_X v
+
(B(X,\nabla_Y Z)+ X(B(Y,Z)) v
\end{align*}
and thus
\begin{align*}
\pi_H(\nabla^{LC}_X \nabla^{LC}_Y Z)
=
\nabla_X \nabla_Y Z + B(Y,Z) \nabla^{LC}_X v.
\end{align*}
Interchanging $X$ and $Y$ gives
\begin{align*}
\pi_H(\nabla^{LC}_Y \nabla^{LC}_X Z)
=
\nabla_Y \nabla_X Z + B(X,Z) \nabla^{LC}_Y v.
\end{align*}
The identity $[X,Y]=\pi_H ([X,Y]) -\omega(X,Y) v$ gives
\begin{align*}
\nabla^{LC}_{[X,Y]}Z
&=
\nabla^{LC}_{\pi_H([X,Y])}Z-\omega(X,Y) \nabla^{LC}_v Z
\\
&=
\nabla _{\pi_H([X,Y])}Z-\omega(X,Y) (\nabla^{LC}_Z v + [v,Z])+
\text{multiple of}\, v
\\
&=
\nabla_{[X,Y]}Z - \omega(X,Y) \nabla^{LC}_Z v + \omega(X,Y) T_v Z
+
\text{multiple of}\, v.
\end{align*}
Thus we get
\begin{align*}
R_g(X,Y)Z
&=
R^{\nabla}(X,Y)Z + B(Y,Z) \nabla^{LC}_X v - B(X,Z) \nabla^{LC}_Y v 
\\
&+ \omega(X,Y) \nabla^{LC}_Z v - \omega(X,Y) T_v Z 
+
\text{multiple of}\, v.
\end{align*}
Now we have
\begin{align*}
g(\nabla_X^{LC}v,W)=-g(v,\nabla^{LC}_X W) = - B(X,W).
\end{align*}
Using the above equation gives
\begin{align*}
&R_g(X,Y,Z,W)
\\
=&
R^{\nabla}(X,Y,Z,W) 
-B(Y,Z)B(X,W)+B(X,Z) B(Y,W)
\\
&-\omega(X,Y)B(Z,W) - \omega(X,Y) g(T_v Z, W)
\\
=&
R^{\nabla}(X,Y,Z,W) 
-B(Y,Z)B(X,W)+B(X,Z) B(Y,W)
\\
&+\frac{1}{2} \omega(X,Y) (\Lv g(Z,W)+\omega(Z,W))
- \frac{1}{2} \omega(X,Y) \Lv g(Z,W) + \frac{f}{2} \omega(X,Y) \omega(Z,W)
\\
=&R^{\nabla}(X,Y,Z,W) 
-B(Y,Z)B(X,W)+B(X,Z) B(Y,W)
+\frac{f+1}{2} \omega(X,Y)\omega(Z,W). 
\end{align*}
\end{proof}

\begin{lemma}
\label{lemma-curvature-id}
For $X,Y,Z,W\in\Gamma(H)$, the curvature tensor $R^{\nabla}$ of $\nabla$ satisfies the following identites:
\begin{align*}
R^{\nabla}(X,Y,Z,W)-R^{\nabla}(Z,W,X,Y)
&=
-
\frac{1}{2}
\Lv g(X,Z) \omega(Y,W) - \frac{1}{2} \omega(X,Z) \Lv g(Y,W)
\\
&+\frac{1}{2} \Lv g(X,W) \omega(Y,Z)+\frac{1}{2}\omega(X,W) \Lv g(Y,Z),
\\
R^{\nabla}(X,Y)Z + R^{\nabla}(Y,Z)X + R^{\nabla}(Z,X)Y
&=
\omega(X,Y) T_v Z + \omega(Y,Z) T_v X + \omega(Z,X) T_v Y
\end{align*}
\end{lemma}
\begin{proof}
By the formula from Lemma \ref{Gauss-equation} and the symmetry of the standard Riemannian curvature tensor $R_g$, we have
\begin{align*}
&R^{\nabla}(X,Y,Z,W)-R^{\nabla}(Z,W,X,Y)
\\
=&
B(Y,Z)B(X,W)-B(X,Z)B(Y,W)-B(W,X)B(Z,Y)+B(Z,X)B(W,Y)
.
\end{align*}
Splitting $B=B^s+B^a$ into a symmetric and alternating part, we get
\begin{align*}
&B(Y,Z)B(X,W)-B(Z,Y)B(W,X)
\\
=&
B^s(Y,Z)B^a(X,W)+B^a(Y,Z)B^s(X,W)
\\
-&
B^s(Z,Y)B^a(W,X)-B^a(Z,Y)B^s(W,X)
\\
=&
2 B^s(Y,Z)B^a(X,W)+2B^a(Y,Z)B^s(X,W).
\end{align*}
Interchanging $X$ and $Y$ gives
\begin{align*}
&B(X,Z)B(Y,W)-B(Z,X)B(W,Y)
\\
=&
2 B^s(X,Z)B^a(Y,W)+2B^a(X,Z)B^s(Y,W).
\end{align*}
Thus we have
\begin{align*}
&R^{\nabla}(X,Y,Z,W)-R^{\nabla}(Z,W,X,Y)
\\
=&
2B^s(Y,Z) B^a(X,W)+2B^a(Y,Z) B^s(X,W)
-2B^s(X,Z) B^a(Y,W)-2B^a(X,Z)B^s(Y,W)
\\
=&
-\frac{1}{2} \omega(X,Z) \Lv g(Y,W)
-\frac{1}{2} \omega(Y,W) \Lv g(X,Z)
+ \frac{1}{2} \omega(X,W) \Lv g(Y,Z)
+\frac{1}{2} \omega(Y,Z) \Lv g(X,W).
\end{align*}
This gives the first identity. To explain the second identity, denote by
$\mathfrak{S}_{X,Y,Z}$ the cyclic sum in $X,Y,Z$. We have the general formula \cite[Theorem 1.24]{besse} 
\begin{align*}
\mathfrak{S}_{X,Y,Z} R^{\nabla}(X,Y)Z
=
\mathfrak{S}_{X,Y,Z} 
(
T(T(X,Y),Z)
+
(\nabla_X T)(Y,Z)
).
\end{align*}
We have
\begin{align*}
T(T(X,Y),Z)
=
\omega(X,Y) T_v Z
\end{align*}
and
\begin{align*}
(\nabla_X T)(Y,Z)
&=
\nabla_X(T(Y,Z))
-T(\nabla_X Y,Z)
-
T(Y,\nabla_X Z)
\\
&=
\nabla_X(\omega(Y,Z)v)
-\omega(\nabla_X Y,Z) v - \omega(Y,\nabla_X Z)v
\\
&=
((\nabla_X \omega)(Y,Z))v
=0.
\end{align*}
This gives the second identity.
\end{proof}

Denote by $R^H$ the restriction of $R^{\nabla}$ to $H$. $R^H$ is a section of 
\begin{align*}
\Lambda^2_H \otimes \Lambda^2_H = S^2(\Lambda^2_H) \oplus \Lambda^2(\Lambda^2_H),
\end{align*}
decomposes as
\begin{align*}
R^{H}=\pi_s(R^{H}) + \pi_a (R^{H})
\end{align*}
into a symmetric and alternating part.
To further discuss the structure of $R^H$ we need the \textit{Kulkarni--Nomizu product}.
For $h,k\in H^*\otimes H^*$ define $h\KN k\in \Lambda_H^2\otimes \Lambda_H^2$ by
\begin{align*}
(h\KN k) (X,Y,Z,W)
=
h(X,W) k(Y,Z) + h(Y,Z) k(X,W)
-h(X,Z) k(Y,W) - h(Y,W) k(X,Z).
\end{align*}
The Kulkarni--Nomizu product allows to interpret the first identity in Lemma \ref{lemma-curvature-id} as
\begin{align}
\label{skew-curvature}
\pi_a(R^{H})
=
\frac{1}{4} \omega\KN \Lv g.
\end{align}
For $\mathcal{R}\in\Lambda_H^2\otimes \Lambda_H^2$ define the \textit{Ricci contraction} $c(\mathcal{R})$ to be
\begin{align*}
c(\mathcal{R})(X,Y) = tr(\mathcal{R}(\cdot,X,Y,\cdot))
=
\sum_{i=1}^4 \mathcal{R}(e_i,X,Y,e_i),
\end{align*}
and the \textit{Bianchi map} $b$ by
\begin{align*}
b(\mathcal{R})(X,Y,Z,W)
=
\frac{1}{3}
(
\mathcal{R}(X,Y,Z,W)
+
\mathcal{R}(Y,Z,X,W)
+
\mathcal{R}(Z,X,Y,W)
).
\end{align*}
For the Bianchi map restricted to $S^2(\Lambda^2_H)$ 
we have $\im\, b = \Lambda_H^4$, which gives the decomposition
\begin{align*}
S^2(\Lambda^2_H) = \mathcal{C}H \oplus \Lambda^4_H,
\end{align*}
where $\mathcal{C}H:= \ker b|_{S^2(\Lambda^2_H)}$ is the space of \textit{algebraic curvature tensors} on $H$.
Thus we have the decomposition 
\begin{align*}
\Lambda^2_H \otimes \Lambda^2_H = \mathcal{C}H \oplus \Lambda^4_H \oplus \Lambda^2(\Lambda^2_H),
\end{align*}
and because $b$ is an idempotent, $R^H$ decomposes as
\begin{align}
R^H = C^H + b(\pi_s(R^H)) + \pi_a(R^H),
\label{curvature-decomposition}
\end{align}
where 
\begin{align*}
C^H= \pi_s(R^H)-b(\pi_s(R^H)).
\end{align*}
A key property of the Kulkarni--Nomizu product is that for $h\in S^2 H^*$ we have $b(h\KN g_H)=0$ and that the resulting map $\cdot \KN g_H: S^2 H^* \rightarrow \mathcal{C}H$ is the adjoint of $4c$. Together with the formula 
$c(h\KN g_H) = 2 h + \mathrm{tr}(h) g_H$ this leads to the decomposition
\begin{align*}
C^H = W^H + \frac{1}{2} r_H\KN g_H - \frac{1}{12} s_H\, (g_H \KN g_H),
\end{align*}
where we call the trace-less part $W^H$ the \textit{Weyl tensor} of $R^H$, and write $r_H:= c(C^H)$ and $s_H:=\mathrm{tr}(c(C^H))$, the contact analogues of Ricci and scalar curvature.

\subsection{Computation of Weitzenb\"ock remainder}

\label{section-Weitzenbock}

In this section we derive a Weitzenb\"ock identity for the Laplacian $\Delta_H$
acting on $\Omega^-_H$.
Denote the component of $W^H$ in $\Lambda_H^-\otimes\Lambda_H^-$ by $(W^H)^-$. We will prove 
\begin{proposition}
We have
\begin{align}
2\Box_H
= 
\nabla_H^*\nabla_H
+(W^{H})^-_{klmn}\varepsilon^k \iota^l \varepsilon^m \iota^n + \frac{s_H}{3}  
- \frac{2}{3} f
\label{W-ID}
\end{align}
\end{proposition}

We first derive a general formula for the curvature term in equation \eqref{general_Bochner_formula}.

\begin{lemma}
\label{general-Bochner-contribution}
$\mathcal{R}\in \Lambda^2_H\otimes\Lambda^2_H$ acts on $2$-forms as 
\begin{align*}
\mathcal{R}_{klmn} \varepsilon^k \iota^l \varepsilon^m \iota^n e^{ab}
=
c(\mathcal{R})_{ka} e^{kb}+ c(\mathcal{R})_{kb} e^{ak}
- \mathcal{R}_{bakm} e^{km}+ 3 b(\mathcal{R})_{bakm} e^{km}.
\end{align*}
\end{lemma}
\begin{proof}
We compute 
\begin{align*}
&\mathcal{R}_{klmn} \varepsilon^k \iota^l \varepsilon^m \iota^n e^{ab}
\\
=&
\mathcal{R}_{klma} \varepsilon^k \iota^l e^{mb} + \mathcal{R}_{klmb} \varepsilon^k \iota^l e^{am}
\\
=&
\mathcal{R}_{klla} e^{kb} - \mathcal{R}_{kbma} e^{km} + \mathcal{R}_{kamb} e^{km} - \mathcal{R}_{kllb} e^{ka}
\\
=&
c(\mathcal{R})_{ka} e^{kb} - c(\mathcal{R})_{kb} e^{ka} + (\mathcal{R}_{kbam}+\mathcal{R}_{akbm})e^{km}
\\
=&
c(\mathcal{R})_{ka} e^{kb}+ c(\mathcal{R})_{kb} e^{ak}
- \mathcal{R}_{bakm} e^{km}+ 3 b(\mathcal{R})_{bakm} e^{km}.
\end{align*}
\end{proof}

The Kulkarni--Nomizu product with $\omega$ behaves differently than that with $g_H$. 

\begin{lemma}
\label{KN-omega}
Let $h\in S^2 H^*$. Then we have
\begin{align*}
b(\omega \KN h) (X,Y,Z,W)
=
2b(\omega\otimes h)(X,Y,Z,W)
=
\frac{2}{3}
\mathfrak{S}_{X,Y,Z}(\omega(X,Y) h(Z,W)),
\end{align*}
and
\begin{align*}
3 b(\omega \KN h)_{bakm} e^{km} = (\omega\KN h)_{bakm} e^{km}.
\end{align*}
\end{lemma}
\begin{proof}
To prove the first formula we calculate
\begin{align*}
&3 b(\omega \KN h)(X,Y,Z,W)
\\
=&
(\omega\KN h)(X,Y,Z,W) + (\omega\KN h)(Y,Z,X,W) +(\omega\KN h)(Z,X,Y,W)
\\
=&
\omega(X,W) h(Y,Z) + \omega(Y,Z) h(X,W)
-\omega(X,Z) h(Y,W) - \omega(Y,W) h(X,Z)  
\\
&+\omega(Y,W) h(Z,X) + \omega(Z,X) h(Y,W)
-\omega(Y,X) h(Z,W) - \omega(Z,W) h(Y,X) 
\\
&+ \omega(Z,W)h(X,Y) + \omega(X,Y) h(Z,W)
- \omega(Z,Y) h(X,W) - \omega(X,W) h(Z,Y) 
\\
=&
2
\{
\omega(X,Y) h(Z,W) + \omega(Y,Z) h(X,W) + \omega(Z,X) h(Y,W)
\}.
\end{align*}
For the second formula we have
\begin{align*}
(\omega \KN h)_{bakm} e^{km}
=
(\omega_{bm} h_{ak}+\omega_{ak} h_{bm}-\omega_{bk} h_{am}-\omega_{am} h_{bk} ) e^{km}
\\
=
2(\omega_{bm} h_{ak}+\omega_{ak} h_{bm}) e^{km}
\end{align*}
and
\begin{align*}
3b(\omega \otimes h)_{bakm} e^{km}
=
(\omega_{ba} h_{km} + \omega_{ak} h_{bm} + \omega_{kb} h_{am}) e^{km}
\\
=
(\omega_{ak} h_{bm} + \omega_{kb} h_{am}) e^{km}
\\
=
(\omega_{ak} h_{bm} + \omega_{bm} h_{ak}) e^{km}.
\end{align*}
Thus proving
\begin{align*}
3 b(\omega\otimes h)_{bakm}e^{km} = \frac{1}{2} (\omega\KN h)_{bakm}e^{km}.
\end{align*}
The statement follows from the first formula.
\end{proof}

We will now compute the contribution to the Weitzenb\"ock identity of each component in the decomposition \eqref{curvature-decomposition}.

\subsubsection*{Contribution of $C^H$}

Here the algebra is the same as on a $4$-dimensional Riemannian manifold. The action of $W^H$ preserves the decomposition $\Lambda_H^2=\Lambda_H^+\oplus \Lambda_H^-$. Denote the component of $W^H$ in $\Lambda_H^-\otimes \Lambda_H^-$ by $(W^H)^-$. Then the total contribution of $C^H$ is
\begin{align}
(W^H)^{-}_{klmn} \varepsilon^k\iota^l\varepsilon^m\iota^n + \frac{s_H}{3}.
\label{contribution-C}
\end{align}

\subsubsection*{Contribution of $b(\pi_s(R^H))$}

Because
\begin{align*}
T_v = \frac{1}{2} \Lv I \circ I -\frac{f}{2} I
\end{align*}
and $\Lv g \sim \Lv I \circ I$, by Lemma \ref{lemma-curvature-id} we have
\begin{align*}
b(R_H)(X,Y,Z,W) = \frac{1}{6} \mathfrak{S}_{X,Y,Z}(\omega(X,Y) \Lv g(Z,W))
-
\frac{f}{6} \mathfrak{S}_{X,Y,Z}(\omega(X,Y) \omega(Z,W)),
\end{align*}
or 
\begin{align*}
b(R_H)
=
\frac{1}{2} b(\omega \otimes \Lv g) -
\frac{f}{2} b(\omega \otimes \omega).
\end{align*}
Therefore by \eqref{skew-curvature} and Lemma \ref{KN-omega},
\begin{align*}
b(\pi_s(R_H))
&=
b(R_H)-b(\pi_a(R_H))
=
\frac{1}{2} b(\omega \otimes \Lv g) -
\frac{f}{2} b(\omega \otimes \omega)
-
\frac{1}{4} b(\omega\KN \Lv g)
\\
&=
-\frac{f}{2} b(\omega \otimes \omega).
\end{align*}
Because $b$ is an idempotem and the contraction is zero on $b(S^2 \Lambda^2) = \Lambda^4$, by Lemma \ref{general-Bochner-contribution} $b(\pi_s(R_H))$ contributes 
\begin{align*}
2 b(\pi_s(R_H))_{bakm} e^{km}
\end{align*}
to the Weitzenb\"ock formula.
\\
We have
\begin{align*}
3b(\omega\otimes\omega)_{bakm}e^{km}
=
\omega_{ba} \omega_{km} e^{km}
+ \omega_{ak} \omega_{bm} e^{km}
+ \omega_{kb} \omega_{am} e^{km}
\\
=
-g_H(e^{ab},\omega) \omega + (I e^a)\wedge(I e^b) -(I e^b)\wedge(I e^a)
=
-g_H(e^{ab},\omega) \omega+ 2 (I\wedge I) e^{ab}.
\end{align*}

\begin{lemma}
$I\wedge I$ acts as the identity on $\Lambda^2_-$.
\end{lemma}
\begin{proof}
\begin{align*}
(I\wedge I) \omega^-_1 = (I\wedge I) (e^0 \wedge e^1-e^2\wedge e^3)
=
e^1\wedge(-e^0)-e^3\wedge(-e^2)=e^0\wedge e^1-e^2\wedge e^3 = \omega_1^-,
\\
(I\wedge I) \omega_2^-
=
(I\wedge I) (e^0\wedge e^2-e^3\wedge e^1)
=
e^1\wedge e^3 - (-e^2)\wedge(-e^0)
=
 e^0\wedge e^2-e^3\wedge e^1=\omega_2^-,
\\
(I\wedge I)\omega_3^-
=
(I \wedge I)(e^0\wedge e^3 - e^1\wedge e^2)
=
e^1\wedge(-e^2) - (-e^0)\wedge e^3
=
 e^0\wedge e^3 - e^1\wedge e^2
 =
 \omega_3^-.
\end{align*}
\end{proof}
Therefore, $b(\pi_s(R_H))$ contributes
\begin{align}
- \frac{2}{3} f
\label{contribution-b}
\end{align}
to the Weitzenb\"ock formula.

\subsubsection*{Contribution of $\pi_a(R_H)$}

\begin{lemma}
We have
\begin{align*}
c(\omega \KN \Lv g) = \mathrm{tr}(\Lv g) \omega.
\end{align*}
\end{lemma}
\begin{proof}
We have
\begin{align*}
\sum_i \omega(X,e_i) e_i
=
\sum_i
g(IX,e_i) e_i 
=
IX
\end{align*}
and thus for $h\in S^2$
\begin{gather*}
c(\omega\KN h)(X,Y)
=
\sum_i
(\omega\KN h) (e_i,X,Y,e_i)
\\
=
\sum_i
\omega(e_i,e_i) h(X,Y) + \omega(X,Y) h(e_i,e_i)
-\omega(e_i,Y) h(X,e_i) - \omega(X,e_i) h(e_i,Y)
\\
=
\mathrm{tr}(h) \omega(X,Y)+h(X,IY)-h(IX,Y) 
.
\end{gather*}
By \eqref{Lie-v-g} for $h=\Lv g$ we have $h(X,IY)-h(IX,Y)=0$.
\end{proof}
Thus the contribution of the Ricci term to the Weitzenb\"ock formula is
\begin{align*}
c(\pi_a(R_H))_{ka} e^{kb}+ c(\pi_a(R_H))_{kb} e^{ak}
=
\frac{1}{4}
\mathrm{tr}(\Lv g)
(\omega_{ka} e^{kb} + \omega_{kb} e^{ak})
=
\frac{1}{4}
\mathrm{tr}(\Lv g)
I(e^{ab}).
\end{align*}
But $I$ acts trivially on $\Lambda^-_H$. 
Thus the contribution is
\begin{align}
-\frac{1}{4} (\omega\KN \Lv g)_{bakm}  e^{km}
+\frac{3}{4} b(\omega\KN \Lv g)_{bakm} e^{km}
=0
\label{contribution-a}
\end{align}
by Lemma \ref{KN-omega}(ii). 
Therefore, the total contribution of $\pi_a(R_H)$ is zero.

Adding the contributions \eqref{contribution-C}, \eqref{contribution-b} and \eqref{contribution-a} gives the identity \eqref{W-ID}.

\bibliographystyle{amsplain}

\bibliography{set-up}

\end{document}